\documentclass[reqno,11pt]{amsart}
\usepackage{amsmath,amsfonts,amssymb,amsthm,color,hyperref,enumerate,mathrsfs,mathtools,todonotes,tikz-cd}
\usepackage[all]{xy}
\usepackage[top=3cm,bottom=3cm,left=2.5cm,right=2.5cm]{geometry}
\usetikzlibrary{chains,decorations.pathreplacing,backgrounds,decorations.markings,babel}
\usepackage{accents}

\usepackage[T1]{fontenc}
\usepackage{bbm}

\newcounter{subeqn} %
\makeatletter
\@addtoreset{subeqn}{equation}
\makeatother

\tikzset{double line with arrow/.style args={#1,#2}{decorate,decoration={markings,%
mark=at position 0 with {\coordinate (ta-base-1) at (0,1pt);
\coordinate (ta-base-2) at (0,-1pt);},
mark=at position 1 with {\draw[#1] (ta-base-1) -- (0,1pt);
\draw[#2] (ta-base-2) -- (0,-1pt);
}}}}

\tikzset{
    labl/.style={anchor=south, rotate=90, inner sep=.5mm}
}

\newtheorem{Theorem}{Theorem}[section]
\newtheorem{Proposition}[Theorem]{Proposition}
\newtheorem{Lemma}[Theorem]{Lemma}

\newtheorem{Corollary}[Theorem]{Corollary}	
\newtheorem{Conjecture}[Theorem]{Conjecture}
\newtheorem{Definition}[Theorem]{Definition}
\newtheorem{Example}[Theorem]{Example}
\newtheorem{Remark}[Theorem]{Remark}

\newtheorem{TheoremA}{Theorem}

\newcommand{\nc}{\newcommand}
\newcommand{\renc}{\renewcommand}

\nc{\fa}{\mathfrak a}
\nc{\fb}{\mathfrak b}
\nc{\fg}{\mathfrak g}
\nc{\fk}{\mathfrak k}
\nc{\fh}{\mathfrak h}
\nc{\ft}{\mathfrak t}
\nc{\fw}{\mathfrak w}
\nc{\fM}{\mathfrak M}
\nc{\CC}{\mathbb{C}}
\nc{\GG}{\mathbb{G}}
\nc{\KK}{\mathbb{K}}
\nc{\NN}{\mathbb{N}}
\nc{\PP}{\mathbb{P}}
\nc{\QQ}{\mathbb{Q}}
\nc{\RR}{\mathbb{R}}
\nc{\ZZ}{\mathbb{Z}}
\nc{\cA}{\mathcal{A}}
\nc{\cB}{\mathcal{B}}
\nc{\cC}{\mathcal{C}}
\nc{\cD}{\mathcal{D}}
\nc{\cE}{\mathcal{E}}
\nc{\cF}{\mathcal{F}}
\nc{\cG}{\mathcal{G}}
\nc{\cI}{\mathcal{I}}
\nc{\cJ}{\mathcal{J}}
\nc{\cK}{\mathcal{K}}
\nc{\cL}{\mathcal{L}}
\nc{\cM}{\mathcal{M}}
\nc{\cN}{\mathcal{N}}
\nc{\cO}{\mathcal{O}}
\nc{\cP}{\mathcal{P}}
\nc{\cQ}{\mathcal{Q}}
\nc{\cR}{\mathcal{R}}
\nc{\cS}{\mathcal{S}}
\nc{\cT}{\mathcal{T}}
\nc{\cU}{\mathcal{U}}
\nc{\cV}{\mathcal{V}}
\nc{\cW}{\mathcal{W}}
\nc{\cX}{\mathcal{X}}
\nc{\cY}{\mathcal{Y}}
\nc{\cZ}{\mathcal{Z}}
\nc{\ba}{\mathbf{a}}
\nc{\bb}{\mathbf{b}}
\nc{\bc}{\mathbf{c}}
\nc{\bd}{\mathbf{d}}
\nc{\bg}{\mathbf{g}}
\nc{\bi}{\mathbf{i}}
\nc{\bj}{\mathbf j}
\nc{\bk}{\mathbf k}
\nc{\bm}{\mathbf{m}}
\nc{\br}{\mathbf{r}}
\nc{\bs}{\mathbf{s}}
\nc{\bt}{\mathbf{t}}
\nc{\bv}{\mathbf{v}}
\nc{\bw}{\mathbf{w}}
\nc{\bz}{\mathbf{z}}
\nc{\bA}{\mathbf A}
\nc{\bB}{\mathbf B}
\nc{\bF}{\mathbf F}
\nc{\bG}{\mathbf G}
\nc{\bL}{\mathbf{L}}
\nc{\bM}{\mathbf{M}}
\nc{\bN}{\mathbf{N}}
\nc{\bP}{\mathbf{P}}
\nc{\bR}{\mathbf{R}}
\nc{\bS}{\mathbf{S}}
\nc{\bT}{\mathbf{T}}
\nc{\bU}{\mathbf{U}}
\nc{\bX}{\mathbf{X}}
\nc{\bY}{\mathbf{Y}}
\nc{\sB}{\mathscr{B}}
\nc{\sC}{\mathscr{C}}
\nc{\Ann}{\operatorname{Ann}}
\nc{\Aut}{\operatorname{Aut}}
\nc{\Coker}{\operatorname{Coker}}
\nc{\Der}{\operatorname{Der}}
\nc{\End}{\operatorname{End}}
\nc{\Ext}{\operatorname{Ext}}
\nc{\fmod}{\operatorname{--fmod}}
\nc{\gr}{\operatorname{gr}}
\nc{\Hom}{\operatorname{Hom}}
\renc{\Im}{\operatorname{Im}}
\nc{\Ind}{\operatorname{Ind}}
\nc{\Ker}{\operatorname{Ker}}
\nc{\MaxSpec}{\operatorname{MaxSpec}}
\nc{\Mod}{\operatorname{Mod}}
\nc{\op}{\operatorname{op}}
\nc{\Rep}{\operatorname{Rep}}
\nc{\Res}{\operatorname{Res}}
\nc{\Span}{\operatorname{Span}}
\nc{\Spec}{\operatorname{Spec}}
\nc{\Sym}{\operatorname{Sym}}
\nc{\Tor}{\operatorname{Tor}}
\nc{\tr}{\operatorname{tr}}
\nc{\val}{\operatorname{val}}
\nc{\wt}{\operatorname{wt}}
\nc{\arxiv}[1]{\href{http://arxiv.org/abs/#1}{\tt arXiv:\nolinkurl{#1}}}
\nc{\Cartan}{\CC[H_\bullet^{(\bullet)}]}
\nc{\con}{\sim}
\nc{\diam}{\diamond}
\nc{\eps}{\varepsilon}
\nc{\fsl}{\mathfrak{sl}}
\nc{\flavour}{G_W}
\nc{\gauge}{G_V}
\nc{\GL}{\operatorname{GL}}
\nc{\Gr}{\mathsf{Gr}}
\nc{\Grml}{\Gr_\mu^{\overline{\lambda}}}
\nc{\hh}{\hslash}
\nc{\id}{\operatorname{id}}
\nc{\leftexp}[2]{\vphantom{#2}^{#1} #2}
\nc{\Lie}{\operatorname{Lie}}
\nc{\nilHecke}{\mathcal{NH}_\bG}
\nc{\sslash}{/\mkern-6mu/}
\nc{\ssslash}{/\mkern-6mu/\mkern-6mu/}
\nc{\Stab}{\mathsf S}
\nc{\hooklongrightarrow}{\lhook\joinrel\longrightarrow}
\nc{\hooklongleftarrow}{\longleftarrow\joinrel\rhook}
\nc{\twoheadlongrightarrow}{\relbar\joinrel\twoheadrightarrow}
\nc{\twoheadlongleftarrow}{\twoheadleftarrow\joinrel\relbar}
\nc{\acom}[1]{\todo[inline,color=green!20]{ Alex: #1 }}

\nc{\CB}{\cA^{\mathsf{sph}}}
\nc{\ICB}{\cA}
\nc{\ACB}{\cA^{\mathsf{ab}}}
\nc{\tCB}{\widetilde{\ICB}^{\mathsf{sph}}}
\nc{\tICB}{\widetilde{\ICB}}
\nc{\CBone}{\CB_{\hbar=1}}
\nc{\ICBone}{\ICB_{\hbar=1}}
\nc{\ACBone}{\ACB_{\hbar=1}}
\nc{\GK}{\bG((z))}
\nc{\GO}{\bG[[z]]}
\nc{\TK}{\bT((z))}
\nc{\TO}{\bT[[z]]}
\nc{\Iwa}{\cI}
\nc{\NK}{\bN((z))}
\nc{\NO}{\bN[[z]]}
\nc{\la}{\lambda}
\nc{\Ya}{{\reflectbox{\rm R}}}
\nc{\Yml}{Y_\mu^\lambda}
\nc{\FYml}{FY_\mu^\lambda}
\nc{\barX}{\bar{X}}
\nc{\barQ}{\bar{Q}}
\nc{\barP}{\bar{P}}
\nc{\excise}[1]{}
\nc{\Pol}{\mathsf{Pol}}
\nc{\PolKLR}{\mathsf{Pol}}
\nc{\yz}{z}
\nc{\YZ}{Z}

\nc{\ev}{\operatorname{ev}}
\nc{\basW}{\eps}
\nc{\basV}{e}
\nc{\DEF}{:=}
\nc{\idem}{\mathsf e}
\nc{\idemCB}{\idem}
\nc{\eV}{w}
\nc{\eW}{z}
\nc{\subrep}[1]{U_{#1}}
\nc{\catT}[1]{\cT_{#1}}
\nc{\catR}[2]{{_{#1} \cR_{#2}}}
\nc{\CCB}[2]{{_{#1}\ICB_{#2}}}
\nc{\unit}{\mathbf{1}}
\nc{\sfu}{\mathsf{u}}
\nc{\shV}{\sfu}
\nc{\dom}{\ZZ^{\bv}_{\textsf{dom}}}
\nc{\laG}{{\boldsymbol{\lambda}}}
\nc{\muG}{{\boldsymbol{\mu}}}
\nc{\laQ}{\lambda}
\nc{\muQ}{\mu}
\nc{\Ymu}{Y_\mu[\eW_1,\ldots,\eW_N]}
\nc{\sD}{\mathscr{D}}

\nc{\bimIK}{\mathcal{P}} 
\nc{\bimKI}{\mathcal{Q}} 
\nc{\Inc}{\mathsf{Inc}}
\nc{\Av}{\mathsf{Av}}

\nc{\objs}{\bt_\RR^{\circ}}
\nc{\paths}{\bt_\RR^{\bullet}}
\nc{\catya}{\sB^\bv}
\nc{\morphism}{\varphi}

\title{Generators for Coulomb branches of quiver gauge theories}
\author{Alex Weekes}
\date{}

\begin{document}
\begin{abstract}
We study the Coulomb branches of $3d$ $\cN=4$ quiver gauge theories, focusing on the generators for their quantized coordinate rings.  We  show that there is a surjective map from a shifted Yangian onto the quantized Coulomb branch, once the deformation parameter is set to $\hbar =1$.  In finite ADE type, this extends to a surjection over $\CC[\hbar]$.  We also show that these algebras are generated by the dressed minuscule monopole operators, for an arbitrary quiver.   Finally, we describe how the KLR Yangian algebra from \cite{KTWWY2} is related to Webster's extended BFN category.  This paper provides proofs for two results which were announced in \cite{KTWWY2}.
\end{abstract}
\maketitle


\section{Introduction}

Let $\bG$ be a connected reductive algebraic group over $\CC$, and $\bN$ a complex representation of $\bG$.  Recently, Braverman, Finkelberg and Nakajima have given a mathematical construction of the Coulomb branch  $\cM_C$ of the $3d$ $\cN =4$ gauge theory associated to $(\bG, \bN)$ \cite{BFN1}, building upon \cite{Nak6}.  $\cM_C$ is an important space studied by physicists, which had escaped a rigourous mathematical definition in any generality.  The theory of Coulomb branches has interesting physical and mathematical applications, for example in symplectic duality \cite{Web3}, \cite{BDGH}.  

The BFN construction realizes $\cM_C$ as an affine algebraic variety, by constructing its coordinate ring as a certain convolution algebra.  Their construction also naturally yields a deformation quantization $\CB$  of the coordinate ring of $\cM_C$, over $\CC[\hbar] = H_{\CC^\times}^\ast(pt)$.  It is interesting to relate these algebras with other more familiar ones, as for example has been done for some Cherednik algebras \cite{KN}, \cite{Web4}, \cite{BEF}.  To do so in general one needs to get a handle on generators, and possibly relations.  An important source for generators are the monopole operators $M_{\laG, f}$ coming from physics; we note that the related {monopole formula} \cite{CHZ} was a primary motivation in Nakajima's foundational work \cite{Nak6}.

In this paper we study the case of {quiver gauge theories}, and their connection with {shifted Yangians}, following \cite{BFN2}.  For a simple quiver (that is, having neither loops nor multiple edges) with vertex set $I$, we fix two $I$--graded vector spaces $W = \bigoplus_{i\in I}W_i$ and $V = \bigoplus_{i\in I} V_i$. To this data we associated a pair $(\bG, \bN)$ by
\begin{align}
\bG &\DEF \prod_{i\in I} \GL(V_i) \label{eq: G} \\
\bN &\DEF \bigoplus_{i\rightarrow j} \Hom_\CC(V_i, V_j) \oplus \bigoplus_{i} \Hom_\CC(W_i, V_i)  \label{eq: N}
\end{align}
We think of $\bG$ as an algebraic group over $\CC$, and $\bN$ its representation in the natural way.  Note that the corresponding {Higgs branch}, meaning the Hamiltonian reduction $T^\ast  \bN \ssslash \bG$, is a Nakajima quiver variety as originally defined in \cite{Nak3}.

If the underlying graph of the quiver is a Dynkin diagram of ADE type, then the Coulomb branch corresponding to $(\bG, \bN)$ is isomorphic to a generalized affine Grassmannian slice $\overline{\cW}_\mu^\lambda$ by \cite[Theorem 3.10]{BFN2} (we ignore a diagram automorphism that was included there).  This is a particularly satisfying result in that it provides a candidate definition for these slices outside of finite type: as Coulomb branches.  See for example \cite{F} for more details on this perspective.

Another approach to affine Grassmannian slices and their quantizations was introduced in work of Kamnitzer, Webster, Yacobi, and the author \cite{KWWY}, via algebras $Y_\mu^\lambda$ called {truncated shifted Yangians}.  This Yangian approach was related with Coulomb branches in \cite[Section 6.8]{BDG}, and further studied in \cite[Appendix B]{BFN2}.  In particular, in finite ADE type with $\mu$ dominant, an isomorphism of algebras was proved in \cite[Corollary B.28]{BFN2}.  In the present paper we extend this result:

\begin{TheoremA} 
\label{thm A}
Let $Q$ be any simple quiver, and $W, V$ as above.
\begin{enumerate}
\item[(a)]  There is an isomorphism of $\CC$--algebras $ Y_\mu^\lambda \cong \CBone$, where $Y_\mu^\lambda$ is a truncated shifted Yangian of type $\fg_Q$, the simply-laced Kac-Moody Lie algebra underlying $Q$.

\item[(b)] In finite ADE type, there is an isomorphism of graded $\CC[\hbar]$--algebras $\mathbf{Y}_\mu^\lambda \cong \CB$ for any $\mu$.
 
\end{enumerate}
\end{TheoremA}

This theorem was announced as \cite[Theorem 4.9]{KTWWY2}, and its proof appears in Section \ref{section: relationship to shifted Yangians}. Put differently, part (a) of the theorem simply says that $\CBone$ is generated by its polynomial subring $H_{\bG}^\ast(pt)$ together with certain dressed minuscule operators $M_{\varpi_{i,1}, f}$ and $M_{\varpi_{i,1}^\ast,f}$. In general type the algebra $Y_\mu^\lambda$ is defined as in \cite[Section 4]{KTWWY2}. We note that this algebra is {\em not quite} built out of the Yangian for $\fg_Q$ as defined for example in \cite[Section 2]{GNW}, unless $Q$ is finite type. See \cite{KTWWY2} for more details.

 In part (b), the algebra $\mathbf{Y}_\mu^\lambda = \operatorname{Rees} Y_\mu^\lambda$ is the Rees algebra with respect to a filtration defined explicitly in terms of the Yangian's generators; our claim is that this filtration agrees with a filtration on $\CBone$ coming from the grading on $\CB$ by homological degree. Part (b) is strictly stronger than part (a) for general $\mu$, because gradings are not bounded below.  We deduce part (b) from \cite[Corollary B.28]{BFN2}, which proves the case of $\mu$ dominant.  One consequence of part (b) is that the coordinate ring of $\cM_C$ is Poisson generated by the dressed miniscule operators $M_{\varpi_{i,1}, f}$ and $M_{\varpi_{i,1}^\ast, f}$, in finite ADE type.

In finite type A, Finkelberg and Tsymbaliuk have proved an analogue of part (b) for K-theoretic Coulomb branches and shifted quantum affine algebras \cite{FT2}.  Their proof overlaps partially with ours (see Proposition \ref{prop: quiver generated by minuscules}), but otherwise makes use of the very different machinery of shuffle algebras.  We believe the rational version of their work proves part (b) of the above theorem in the finite type A case (i.e.~for ordinary homological Coulomb branches).  Conversely it should be possible to extend results from the present paper to the K-theoretic setting, with appropriate changes.

\subsection{Our approach}
We make use of a diagram of algebras
\begin{equation}
\label{intro 1}
\CB \longrightarrow \ICB \longleftarrow \ACB
\end{equation}
They are all variants on the BFN construction.  This is a generalization of the case of a pure gauge theory (meaning $\bN = 0)$, which may be more familiar, where the above may be identified with a diagram
\begin{equation*}
\operatorname{Toda}_\hbar(\bG^\vee) \longrightarrow \mathbb{H}_\hbar \longleftarrow \mathscr{D}_\hbar(\bT^\vee),
\end{equation*}
relating the quantized open Toda lattice for the Langlands dual group, the degenerate nil-DAHA $\mathbb{H}_\hbar$, and the $\hbar$--differential operators on the dual torus $\bT^\vee$.   These algebras are studied for example in \cite{Ginzburg4}, \cite{Lon2}.   

In general, a useful analogy for (\ref{intro 1}) to keep in mind comes from the study of quotient singularities: if $\Sigma$ is a finite group acting on a $\CC$--algebra $A$, then we may consider 
\begin{equation}
\label{intro 2}
A^\Sigma \longrightarrow A \# \Sigma \longleftarrow A
\end{equation} 
This analogy is quite strong\footnote{With good reason: denoting $\bT \subset \bG$ a maximal torus and $\bT^\vee$ its dual, we can take  $A = \CC[T^\ast \bT^\vee]$ and $\Sigma$ the Weyl group. Then (\ref{intro 2}) is a ``classical approximation'' of (\ref{intro 1}), which neglects quantum corrections.  See the introduction of \cite{BDG} for a discussion.}: $\ICB$ is in many respects better behaved than $\CB$, and it contains $\CB \cong \idem \ICB \idem$ as a spherical subalgebra.  In fact, $\ICB$ is a matrix algebra over $\CB$ as shown by Webster \cite{Web3} (see  Section \ref{sec: CB and ICB} below).  We prove that, at least in the case of quiver gauge theories, the nilHecke algebra $\nilHecke$ plays the role of the group algebra $\CC[\Sigma]$:
\begin{TheoremA}[Corollary \ref{cor: quiver ICB generated by ACB and nilHecke} in the main text]
\label{thm B}
For any quiver gauge theory, $\ICB$ is generated by its subalgebras $\ACB$ and $\nilHecke$.
\end{TheoremA}

This gives a positive answer to a conjecture of Dimofte and Garner \cite[Section 3.5]{DimGar}.  In a sense, it tells us that $\ICB $ is some sort of smash product of $\ACB$ and $\nilHecke$.  

Along the way to proving this result, we explore other aspects of the relationship  between the algebras $\CB, \ICB, \ACB$ which may be of independent interest (Sections \ref{sec: CB and ICB} and \ref{sec: generation by minuscules general case}).  Ultimately, a main ingredient in the proof is the fact that $\CB$ is generated by minuscules (Proposition \ref{prop: quiver generated by minuscules}).  This result is known to experts \cite[Remark 6.7]{BFN1}, \cite[Section 4.3]{BDG}.  It is also proven during the course of the proof of \cite[Theorem 4.29]{FT2}.  

\subsection{Webster's BFN category, cylindrical KLR diagrams, and Yangians}

In the final section of this paper we focus on the extended BFN category $\sB$ introduced by Webster \cite[Section 3]{Web3}, specialized to the situation of a quiver gauge theory $(\bG, \bN)$ with simple $Q$ as defined above.  We consider a full subcategory $\catya$ of $\sB$, whose objects are labelled by certain tuples $\bi = (i_1,\ldots,i_n) \in I^\bv$, see Section \ref{section: a subcategory}.   By repackaging \cite[Theorem 3.10]{Web3}, we describe how morphisms in this category can be encoded in terms of cylindrical KLR diagrams.  This is similar to \cite{Web4}, which covers the Jordan quiver case, as well as forthcoming work \cite{Web5}.

One of the key players in \cite{KTWWY2} is the {KLR Yangian algebra} $\Ya$, which is defined in terms of the same cylindrical KLR diagrams.  We consider $\Ya$ as an algebra over $\CC[\hbar]$.  It contains idempotents $\idem(\bi)$ for each $\bi \in I^\bv$. For a certain choice $\bi_\bv \in I^\bv$  there is a primitive idempotent $\idemCB\in \Ya$ which is a summand of $\idem(\bi_\bv)$. 

\begin{TheoremA}[\mbox{Section \ref{section: finale} in the main text}]
\label{thm C}
\

\begin{enumerate}
\item[(a)]  There is an isomorphism of algebras
$$
\bigoplus_{\bi,\bj\in I^\bv} \idem(\bj) \Ya \idem(\bi)  \stackrel{\sim}{\longrightarrow} \bigoplus_{\bi, \bj \in I^\bv} \Hom_{\catya}(\bi, \bj)
$$

\item[(b)] There is an isomorphism $Y_\mu^\lambda \cong \idemCB \Ya_{\hbar =1} \idemCB$.  In finite ADE type, this upgrades to an isomorphism $\mathbf{Y}_\mu^\lambda \cong \idemCB \Ya \idemCB$.
\end{enumerate}
\end{TheoremA}

As discussed above, part (b) is the truly new content here, while part (a) is due to Webster. Part (b) completes the proof of \cite[Theorem 4.19]{KTWWY2}. 

In the case when the framing $W= 0$, the algebra in part (a) of the theorem above is considered in \cite[Section 6.5]{F}, where it is denoted by $\mathcal{H}_V$, and some interesting conjectures are made about its properties (which we do not address in this paper, but see Remark \ref{remark: finkelberg icm}).  With $W$ is arbitrary, we get a generalization of $\mathcal{H}_V$.

\subsection{OGZ algebras}
As shown in \cite[Appendix A(i)--(ii)]{BFN2}, there is an embedding $\bz^\ast(\iota_\ast)^{-1}: \CBone \hookrightarrow \sD$ into a ring $\sD$ of difference operators\footnote{Note that we have specialized $\hbar =1$ here to match with OGZ algebras, but this is not necesary for this embedding.}.  This map generalizes the GKLO representation from \cite{KWWY}, which was based on a representation of the Yangian defined in \cite{GKLO}.  

Part (a) of Theorem \ref{thm A} tells us that the image of $\CBone$ in $\sD$ is the subalgebra generated by a certain collection of difference operators (\ref{eq: difference ops 1}--\ref{eq: difference ops 3}).  Suppose that $Q$ corresponds to a finite type A Dynkin diagram with orientation $1\leftarrow \cdots \leftarrow n-1$, and $W = W_{n-1}$. After a change of variables, the above difference operators in this case are precisely the generators of an {\em Orthogonal Gelfand-Zetlin} algebra $U(\mathbf{r})$, as defined by Mazorchuk \cite{Maz}.  In particular, Theorem \ref{thm A} says that there is an isomorphism
\begin{equation}
\CBone \stackrel{\sim}{\longrightarrow} U(\mathbf{r})
\end{equation}
See Section \ref{section: OGZ} for more details.    From this perspective, one can think of $\CBone$ for an arbitrary $Q$ as a sort of generalized OGZ algebra.  

This connection between OGZ algebras and quantized Coulomb branches has been exploited to give new results on the representation theory of OGZ algebras \cite{Web6}, \cite[Section 6]{KTWWY2}.

\subsection{Loops and multiple edges}

Although we focus on simple quivers, many of our results apply to quivers with multiple edges and/or loops, with minimal changes.  Essentially, the exceptions to this rule are those statements about Yangians.  

More precisely, the results in Section \ref{sec: geometry} are valid for arbitrary Coulomb branches, while the results in Section \ref{section: generators in the quiver gauge theory case} are valid for arbitrary quivers, with the exception of Section \ref{section: relationship to shifted Yangians}.  What is currently lacking is an appropriate generalization of shifted Yangians to this setting. Note that the Jordan quiver is related to the Yangian of $\widehat{\mathfrak{gl}}(1)$ \cite{KN}, and in general it seems reasonable the appropriate generalization should come from shuffle or cohomological Hall algebras.  Similarly, the results in Section \ref{sec: coulomb cat} can be generalized: Webster's work \cite{Web3} applies to arbitrary Coulomb branches.

Finkelberg and Goncharov study multiloop versions of the Jordan quiver, having $r$ loops instead of one.  For certain dimension vectors they realize the Coulomb branch as a Slodowy slice for $\mathfrak{sp}(2r)$ \cite{FG}, and they expect that the quantized Coulomb branch is isomorphic to the corresponding finite W--algebra.  It would be interesting to understand what the analogue of the shifted Yangian is in this case, and how it relates to $U(\mathfrak{sp}(2r) )$.

\subsection{Relation to other work}

As was discussed above, the connection between Yangians and Coulomb branches was studied in \cite{BDG}, \cite{BFN2}.  This built on relations between affine Grassmannian slices and Yangians introduced in \cite{KWWY}, which was in turn inspired by work of Gerasimov, Kharchev, Lebedev and Oblezin relating Yangians with monopoles \cite{GKLO}.   Finkelberg and Rybnikov have studied the related topic of quantizations of Zastava spaces using Borel Yangians \cite{FR1}, \cite{FR2}, which is also closely connected to the present work, see Remark \ref{remark: Zastava}.

In the case of the Jordan quiver, Kodera and Nakajima \cite{KN} prove that the quantized Coulomb branch is a subquotient of the Yangian of $\widehat{\mathfrak{gl}}(1)$.  This is the appropriate analogue of Theorem \ref{thm A} in this case. The Jordan case is also studied in \cite{Web4}, \cite{BEF}, and connected to Cherednik algebras.

There are several papers studying related questions in the physics literature, including \cite{BDG}.  The work of Dimofte and Garner \cite{DimGar} is close to our approach.  They study the (quantized) coordinate rings for Coulomb branches star-shaped quivers, and in particular also make use of divided difference operators.  Hanany and Miketa study certain quivers of finite AD type \cite{HM}, and also see what we believe are part of the Yangian's relations. We are hopeful that the results and techniques of the present paper will prove useful in further physical research.

Finkelberg and Tsymbaliuk have developed the parallel situtation of K-theoretic Coulomb branches for quiver gauge theories, and their relationship to shifted quantum affine algebras \cite{FT1}, \cite{FT2}, which are algebras they introduced.  As mentioned above, they have recently proven the analog of Theorem \ref{thm A} in finite type A, by making use of shuffle algebras.  

Finally, Cautis and Williams \cite{CW} have studied some questions related to ours, but one categorical level higher: they study (perverse) coherent sheaves on $\Gr_{GL(n)}$.  Put differently, they study the categorification of the (K-theoretic) Coulomb branch for pure $\GL(n)$ gauge theory.

\subsection{Some conventions}
\label{sec: notation} 
Let $G$ be a group and $H\subset G$ its subgroup.  For $V$ a representation of $H$, we will denote the induced vector bundle over $G/H$ by $G\times^H V = (G\times V) / H$. It is the quotient by the $H$--action $h\cdot (g, v) = (g h^{-1}, hv)$.  

\subsubsection{Quiver gauge theories}
As above, we will fix a quiver $Q$ with vertex set $I$. For brevity, we will write $i\rightarrow j$ if and only if there is an arrow from $i$ to $j$.  We write $i\sim j$ if $i\rightarrow j$ or $j\rightarrow i$.  We assume that $Q$ is simple: there is at most one arrow connecting $i$ and $j$, and there are no loops $i\rightarrow i$.  

Fix $I$--graded vector spaces $W = \bigoplus_{i\in I}$, $V = \bigoplus_{i\in I} V_i$.  Throughout the paper $(\bG, \bN)$ will be defined as in (\ref{eq: G}) and (\ref{eq: N}), although the results of Section \ref{sec: geometry} are valid for arbitrary $(\bG, \bN)$.  We will also let $\bF = \prod_i \bF_i \subseteq \prod_{i\in I} \GL(W_i)$ be a maximal torus.  This group acts on $\bN$, commuting with the action of $\bG$.  Our choice of $\bF$  is a matter of convenience, and with minor changes our results hold for the full flavour symmetry group $\bF = N^\circ_{\GL(\bN)}(\bG)/ \bG$ as considered  in \cite[Section 2]{Web3}.

Write $\bv_i = \dim_\CC V_i$, and also $|\bv| = \sum_i \bv_i$.  For each $i\in I$, we will fix an ordered basis $\{\basV_{i,r}: 1\leq r\leq \bv_i\}$ for $V_i$.  Let $\bT \subset \bG$ be the corresponding maximal torus consisting of diagonal matrices at each node, and $\bB, \bB^- \subset \bG$ be the Borel subgroups given by upper/lower triangular matrices at each node.  We denote $\bg = \operatorname{Lie} \bG$ and $\bt = \operatorname{Lie} \bT$.

The Weyl group $\Sigma$ of $\bG$ is a product of symmetric groups $\Sigma = \prod_i \Sigma_{\bv_i}$. The coweight lattice of $\bG$ and $\bT$ is $\ZZ^{\bv} = \bigoplus_{i\in I} \ZZ^{\bv_i}$.  For our choice of $\bB$, recall that a coweight $\laG = (\laG_{i,r})_{i\in I, 1\leq r\leq \bv_i}$ is dominant iff $\laG_{i, r} \geq \laG_{i,s}$ for all $i\in I$ and $1\leq r \leq s \leq \bv_i$.  The extended affine Weyl group of $\GL(V_i)$ is the semidirect product $\widehat{\Sigma}_{\bv_i} = \ZZ^{\bv_i} \rtimes \Sigma_{\bv_i}$, while the extended affine Weyl group associated to $\bG$ is the product $\widehat{\Sigma} = \prod_{i\in I} \widehat{\Sigma}_{\bv_i}$.   For $\laG \in \ZZ^\bv$ we will sometimes write $z^\laG \in \widehat{\Sigma}$, or else just $\laG \in \widehat{\Sigma}$.

\subsubsection{(Co)homology}
\label{notation: cohomology}
Corresponding to our basis of $V$, there are coordinates on $\bt$ which we denote $\eV_{i,r}$.  We identify the cohomology rings
\begin{align*}
H_{\bT}^\ast(pt) = \CC[ \eV_{i,r} : i \in I, 1\leq r\leq \bv_i], \qquad H_{\CC^\times}^\ast(pt)  = \CC[\hbar],
\end{align*}
so that $\hbar = c_1(\CC)$ for the weight 1 action of $\CC^\times$ on $\CC$.  Following the notation of \cite[Section 3(v)]{BFN2} we will also choose coordinates on $\operatorname{Lie}\bF$ compatibly with the decomposition $\bF = \prod_i \bF_i$, and identify $H_\bF^\ast(pt) = \CC[z_1,\ldots,z_N]$.  We can think of $z_1,\ldots,z_N$ as corresponding to a sequence $i_1,\ldots, i_N$ where $i$ appears $\bw_i$ times, so that  $H_{\bF_i}^\ast(pt) = \CC[z_k : i_k = i]$.

Throughout this paper, we implicitly use the notion of Borel-Moore homology of a placid ind-scheme with a dimension theory \cite[Section 3.9]{Lon}. We will always take $\CC$--coefficients, although some results will continue to hold with coefficients in an arbitrary commutative ring\footnote{We should mention that the consideration of positive characteristic coefficients has led to several interesting results about Coulomb branches, for example in \cite{Lon} and \cite{McWeb}.}.  This is an appropriate formalism for all computations below, and one can verify compatibility of dimension theories as required.  For our purposes, it suffices to know that there are versions of all usual operations on equivariant Borel-Moore homology (e.g.~proper pushforward, flat pull-back, restriction with supports/refined Gysin homomorphisms) which interact as expected (e.g.~proper base change).  In other words, one can more or less pretend to be working in the context of \cite{Fulton}. 

\subsection*{Acknowledgements}  I am grateful to Joel Kamnitzer and Ben Webster for many helpful conversations and suggestions, and to our collaborators Peter Tingley and Oded Yacobi for our work together which precipitated this paper.   I also would like to thank Tudor Dimofte, Michael Finkelberg, Dinakar Muthiah, and Alexander Tsymbaliuk for discussions and comments.
This research was supported in part by Perimeter Institute for Theoretical Physics. Research at Perimeter Institute is supported by the Government of Canada through the Department of Innovation, Science and Economic Development Canada and by the Province of Ontario through the Ministry of Research, Innovation and Science.

\section{Quantized Coulomb branches}
\label{sec: geometry}

In this section we overview the theory of Coulomb branches from \cite{BFN1}, as well some results from \cite{Web3}.  We then study relationships between elements of $\CB, \ICB$ and $\ACB$.

For simplicity of notation, we will work throughout with $(\bG, \bN)$ a quiver gauge theory.  However, modulo  minor notational differences, all results in this section hold for arbitrary $(\bG, \bN)$. 

\subsection{The affine Grassmannian and affine flag variety}
\label{sec: Gr and Fl}

Consider the loop groups $\GO \subset \GK$, as well as the Iwahori subgroup $\Iwa \subset \GO$ defined by
\begin{equation}
\label{eq: the Iwahori}
\Iwa \DEF \left\{ g(z) \in \GO : g(0) \in \bB^- \right\}
\end{equation}
Recall that the affine Grassmannian and affine flag variety for $\bG$ are defined as the quotient spaces
\begin{equation*}
\Gr \DEF \GK / \GO, \qquad \cF \DEF \GK / \Iwa
\end{equation*}
These objects may be thought of as ind-schemes over $\CC$, see for example \cite{Zhu} for an overview.  We will mostly work on the level of $\CC$-points of $\Gr$ or $\cF$, and by abuse of notation will simply write $\GO$ rather than $\bG( \CC[[z]] )$, etc.  We denote a point of $\Gr$ or $\cF$ by a class $[g]$, where $g\in \GK$. 

There is an embedding of $\widehat{\Sigma}$ into $\cF$, and of $\Sigma^\bv$ into $\Gr$.  For $\bG = \prod_i \GL(V_i)$ we can even embed $\widehat{\Sigma}$ into $\GK$: to $\laG \in \ZZ^\bv$ we associate the point $z^\laG = \prod_i \operatorname{diag}(z^{\laG_{i, 1}},\ldots, z^{\laG_{i, \bv_i}} )$, while any $w\in \Sigma$ may be embedded as its permutation matrix $w \in \bG \subset \GK$.  In any case, we will only ever use their images in $z^\laG \in \Gr$ and $w z^\laG \in \cF$ for $\laG \in \ZZ^\bv, w\in \Sigma$.

There is a stratification of $\Gr$ by the $\GO$--orbits $\Gr^\laG = \GO z^\laG$ for $\laG$ dominant.  Similarly, $\cF$ is stratified by the $\Iwa$--orbits $\cF^w = \Iwa w$ (Schubert cells) for $w\in \widehat{\Sigma}$. These satisfy the closure relations
\begin{equation*}
\overline{\Gr^{\laG}} = \bigsqcup_{\muG \leq \laG} \Gr^{\muG}, \qquad \overline{\cF^w} = \bigsqcup_{v \leq w} \cF^v,
\end{equation*}
where $\leq$ denotes the affine Bruhat order. In particular, $\muG \leq \laG$ iff $\laG-\muG$ is a non-negative sum of simple coroots of $\bG$.

\subsection{Quantized Coulomb branch algebras} 
\label{section: quantized Coulom branch algebras}
We consider three variations on the quantized Coulomb branch algebra from \cite{BFN1}.  To this end, consider the spaces
\begin{align}
\cR^{\mathsf{sph}} &\DEF \left\{ [g, s] \in \GK \times^{\GO} \NO : gs \in \NO \right\}, \label{eq: R CB def}\\
\cR & \DEF \left\{  [g, s] \in \GK \times^{\Iwa} \NO : gs \in \NO \right\}, \label{eq: R ICB def}\\
\cR^{\mathsf{ab}} &\DEF  \big\{ ( [t, n]) \in \TK \times^{\TO}\NO : tn \in \NO  \big\} \label{eq: R ACB def}
\end{align}
	Each may be considered as a placid ind-scheme of infinite type over $\CC$. Following \cite[Section 2(ii)]{BFN1}, we can define the equivariant Borel-Moore homology
\begin{align}
\CB & \DEF H_\ast^{ ( \GO \times \bF) \rtimes \CC^\times} ( \cR^{\mathsf{sph}} ), \label{eq: CB def}\\
\ICB & \DEF H_\ast^{(\Iwa \times \bF) \rtimes \CC^\times} ( \cR ), \label{eq: ICB def}\\
\ACB & \DEF H_\ast^{(\TO \times \bF) \rtimes \CC^\times} (\cR^{\mathsf{ab}} ) \label{eq: ACB def}
\end{align}

Note that in (\ref{eq: CB def})--(\ref{eq: ACB def}), the group $\CC^\times$ is acting by loop rotation on the spaces (\ref{eq: R CB def})--(\ref{eq: R ACB def}), respectively.  Following \cite{BFN1}, we will modify this action slightly by letting $\CC^\times$ also act on $\bN$ by weight $1/2$.  Explicitly, for a point $[g(z), s(z)]$ in any of these respective spaces, the action of $\tau\in \CC^\times$ is by $$\tau\cdot [ g(z), s(z) ] = [ g(\tau z), \tau^{1/2} s(\tau z) ]$$
Meanwhile, the group $\GO \times \bF$ acts on $\cR^{\mathsf{sph}}$ by  $(g_1, f) \cdot [g_2, s] = [g_1 g_2, f s]$, and similarly for $\cR, \cR^{\mathrm{ab}}$.

There is a convolution product $\ast$ on the above algebras, defined as in \cite[Section 3(iii)]{BFN1}, or alternatively as in \cite[Appendix B(ii)(d)]{BFN3}.  The Coulomb branch construction comes equipped with the following features:
\begin{Theorem}[\mbox{\cite{BFN1}}]
\label{thm: bfn thm}

\begin{itemize}
\item[(a)] $\CB$ is an associative algebra under the convolution product $\ast$, with unit $\unit\in \CB$ the fundamental class of the fibre of $\cR^{\mathsf{sph}} \rightarrow \Gr$ over the identity point $[1] \in \Gr$.  

\item[(b)] The convolution product is $H_{\bG\times \bF\times \CC^\times}^\ast(pt)$--linear in the first variable with respect to cap product. In particular, there is an algebra embedding $H_{\bG\times\bF \times \CC^\times}^\ast(pt) \hookrightarrow \CB$ defined by $\varphi \mapsto \varphi \cap \unit$

\item[(c)] $\CB$ is free as a left (or right) module over $H_{\bG\times \bF\times \CC^\times}^\ast(pt)$.
\end{itemize}

Analagous claims hold for $\ICB$ and $\ACB$, where in parts (b), (c) the relevant algebra is $H_{\bT\times \bF \times \CC^\times}^{\ast}(pt)$.
\end{Theorem}

We call $\CB$ the {\bf quantized Coulomb branch algebra} associated to the pair $(\bG, \bN)$.  Its base change $\CB_{\hbar = 0}$ at the irrelevant ideal $H_{\CC^\times}^\ast(pt) \rightarrow \CC$ is a commutative ring, and
\begin{equation*}
\cM_C \DEF \Spec \CB_{\hbar =0}
\end{equation*}
is called the {\bf Coulomb branch} associated to $(\bG, \bN)$ \cite[Definition 3.13]{BFN1}.  (More precisely, it is the {\em flavour deformation} of the Coulomb branch corresponding to $\bF$, see \cite[Section 3(viii)]{BFN1}).    Similarly, $\ACB$ is the quantized Coulomb algebra branch associated to $(\bT, \bN)$.  The algebra $\ICB$ is not strictly speaking a quantized Coulomb branch algebra, but rather is a matrix algebra over $\CB$ (see Theorem \ref{thm: ICB matrix algebra over CB}). $\ICB_{\hbar = 0}$ can thus be thought of as the endomorphism algebra of a (trivial) vector bundle over $\cM_C$.


\subsection{Filtrations and monopole operators}
\label{section: monopole}

There is a natural map $\cR^{\mathsf{sph}}\rightarrow \Gr$ defined by $[g,s] \mapsto [g]$.  Denote by $\cR^{\mathsf{sph}}_{\leq \laG}$ the preimage of $\overline{\Gr^{\laG}}$ under this map.  By definition,
\begin{equation}
\label{eq: CB filtration}
\CB = \lim_{\longrightarrow} H_\ast^{\GO \rtimes \CC^\times \times \bF}( \cR^{\mathsf{sph}}_{\leq \laG} )
\end{equation}
where this directed system is defined with respect to proper pushforward maps. This is an algebra filtration \cite[Section 6(i)]{BFN1}.  
For $\laG$ minuscule (or equivalently, such that $\overline{\Gr^\laG} = \Gr^\laG$)  there are well-defined classes
\begin{equation}
\label{eq: def of monopole}
M_{\laG, f} \DEF f \cap [ \cR^{\mathsf{sph}}_{\leq \laG} ]  \in H^{(\GO\times \bF)\rtimes \CC^\times}_\ast(\cR^{\mathsf{sph}}_{\leq \laG} )  
\end{equation}
which we call {\bf dressed minuscule monopole operators}. Here the ``dressing'' $ f\in  H^\ast_{\bT \times \bF\times \CC^\times}(pt)^{\Sigma_\laG}$,  where $\Sigma_\laG \subset \Sigma$ is the stabilizer of $\laG$ and acts only on the $H_\bT^\ast(pt)$ part.  See  \cite[Section 6(ii)]{BFN1} for more details.

\begin{Remark}
\label{remark: monopoles not welldef}
It is predicted from physics that $\CB$ has a basis given by {dressed monopole operators} $M_{\laG, f}$ for {all} dominant $\laG$, not necessarily minuscule \cite{BDG}.  Such elements are not canonically defined in general \cite[Remark 6.5]{BFN1}, although they are well-defined in the associated graded with respect to the above filtration.     
\end{Remark}

Similarly, there is a map $\cR \rightarrow \cF$, and we denote by $\cR_{\leq w}$ the preimage of Schubert variety $\overline{\cF^w}$. $\ICB$ is a filtered algebra, analogously to (\ref{eq: CB filtration}), and in this case the fundamental classes $[\cR_{\leq w} ]$ form a basis for $\ICB$ over $H_{\bT \times \bF \times \CC^\times}^{\ast}(pt)$ as a left (or right) module.

Finally, there is a map $\cR^{\mathsf{ab}} \rightarrow \Gr_{\bT} = \TK / \TO$, and for $\laG \in \ZZ^\bv$ we define $\cR^{\mathsf{ab}}_\laG$ to be the preimage of the point $z^\laG \in \Gr_\bT$.  We will denote the corresponding fundamental class by $r_\laG \DEF [ \cR^{\mathsf{ab}}_\laG]$.  These elements form a basis for $\ACB$ as a left (or right) module over $H_{\bT\times \bF \times \CC^\times}^\ast(pt)$.  There is an explicit formula $ r_\laG r_\muG = A(\laG, \muG) r_{\laG+\muG}$ for their structure constants, given in \cite[Theorem 4(iii)]{BFN1}.

\subsection{Relationship between \texorpdfstring{$\ACB$ and $\ICB$}{Aab and A}}
\label{section: ACB and ICB}
Since there is a surjection $\Iwa\rightarrow \TO$ with pro-unipotent kernel, the inclusion $\cR^{\mathsf{ab}} \hookrightarrow \cR$ induces a pushforward map
\begin{equation*}
\ACB = H_\ast^{(\TO \times \bF)\rtimes \CC^\times}(\cR^{\mathsf{ab}}) \longrightarrow H_\ast^{(\TO\times \bF)\rtimes\CC^\times} (\cR) \cong H_\ast^{(\Iwa\times \bF)\rtimes\CC^\times}(\cR) = \ICB
\end{equation*}
This is an injective algebra homomorphism, see \cite[Section 3]{Web3} (c.f.~also \cite[Section 5(iii)]{BFN1}).

\subsection{The nilHecke algebra}
\label{sec: nilhecke algebra}
Recall that there is a natural action of the Weyl group $\Sigma$ on $H_\bT^\ast(pt)$, and it is a free module over its invariant subalgebra $H_\bG^\ast(pt) = H_\bT^\ast(pt)^\Sigma$ of rank $|\Sigma|$. The {\bf nilHecke algebra} for $\bG$ may be defined as 
\begin{equation*}
\nilHecke \DEF \End_{H_\bG^\ast(pt)} ( H_\bT^\ast(pt) )
\end{equation*}
See \cite[Section XI]{Kumar}, \cite{Dem}, or \cite[Section 7.1]{Ginzburg4}.  $\nilHecke$ is a matrix algebra over $H_\bG^\ast(pt)$ of rank $|\Sigma|$, and is generated by $H_\bT^\ast(pt)$ (acting on itself by multiplication) together with the divided difference operators $\partial_{i,r} $ for $i\in I$ and $1\leq r < \bv_i$.  These are the operators on $H_\bT^\ast(pt)$ defined by
\begin{equation}
\label{eq: localizaiton of divided difference}
\partial_{i,r} = \frac{1}{\eV_{i,r} - \eV_{i,r+1} } (1-s_{i,r})
\end{equation}
where $s_{i,r} \in \Sigma_{\bv_i}$ is a simple transposition.   These operators satisfy the braid relations of $\Sigma$, so in particular there is a well-defined element $\partial_w\in\nilHecke$ for any $w\in \Sigma$.  There is an obvious inclusion $\CC[\Sigma] \subset \nilHecke$, since $s_{i,r} = (\eV_{i,r} - \eV_{i,r+1}) \partial_{i,r}  + 1$.

As is well-known, the nilHecke algebra may also be realized as the homology $H^{\bB^-}_\ast(\bG /\bB^-)$, endowed with the convolution product $\alpha \ast \beta \DEF m_\ast ( q^\ast)^{-1} p^\ast( \alpha \otimes \beta)$ defined using the diagram
\begin{equation}
\label{eq: conv for G/B}
\bG/ \bB^- \times \bG / \bB^- \stackrel{p}{\longleftarrow} \bG \times \bG / \bB^- \stackrel{q}{\longrightarrow} \bG\times^{\bB^-} \bG / \bB^- \stackrel{m}{\longrightarrow} \bG / \bB^-
\end{equation}
analogously to \cite[Section 3]{BFN1}.  For each $w\in \Sigma$ there is a Schubert variety $X^w = \overline{\bB^- w \bB^- / \bB^-} \subset \bG / \bB^-$.  The element $\partial_{i,r}$ corresponds to the fundamental class $[X^{s_{i,r}}]$, and more generally $\partial_w = [X^w]$ for any $w\in \Sigma$, as follows from the study of Bott-Samelson resolutions.  The equality (\ref{eq: localizaiton of divided difference}) expresses the localization theorem at $\bT$--fixed points. Note that localization can be thought of as an algebra homomorphism into the smash product $\operatorname{Frac}(H_\bT^\ast(pt) ) \# \Sigma$.

For each simple reflection $s_{i,r}$ there is also a fundamental class $[\cR_{\leq s_{i,r}}] \in \ICB$, as in Section \ref{section: monopole}.  By abuse of notation we will denote this class by $\partial_{i,r}$, as is justified by the following:
\begin{Lemma}
Together with $H_{\bT}^\ast(pt) \subset \ICB$, the cycles $\partial_{i,r} = [\cR_{\leq s_{i,r}}] \in \ICB$ generate a copy of the nilHecke algebra.  Moreover, $\partial_w = [\cR_{\leq w}]$ for any $w\in \Sigma$.
\end{Lemma}
\begin{proof}
There is an algebra embedding $\bz^\ast: \ICB \hookrightarrow H_\ast^{(\Iwa\times \bF)\rtimes \CC^\times}(\cF$), as in \cite[Lemma 5.11]{BFN1}.  For any element $w\in \Sigma$ of the finite Weyl group, it is not hard to see that  $\cR_{\leq w}$ is the full preimage of $\overline{\cF^{w}}$ under the natural map $\GK \times^{\Iwa} \NO \rightarrow \cF$ (and in fact $\cR_{\leq w} \cong \overline{\cF^w}\times \NO$). By the construction of $\bz^\ast$ it follows that $\bz^\ast([ \cR_{\leq w}]) = [ \overline{\cF^{w}} ]$.  The cycle $\overline{\cF^w}$ is the image of $X^w$ under the natural embedding of the finite flag variety $\bG / \bB^- \subset \cF$.  This embedding is compatible with convolution, and therefore the elements $[\overline{\cF^w}]$ give a copy of the nilHecke algebra. Since $\bz^\ast$ is injective, this proves the claim. 
\end{proof}

Consequently there is a chain of  embeddings $\CC[\Sigma] \subset \nilHecke \subset\ICB$, and so $\ICB$ contains the symmetrizing idempotent $\idemCB\in\CC[\Sigma]$.  Geometrically, this can be written alternatively as
\begin{equation}
\label{eq: geometric idempotent}
\idemCB = \frac{1}{|\Sigma|} \operatorname{eu}(T) \cap [ \cR_{\leq w_0} ] = (-1)^{\ell(w_0)}\frac{1}{|\Sigma|} [ \cR_{\leq w_0}] \ast \Delta
\end{equation}
as follows from a standard nilHecke algebra calculation (e.g.~by the localization theorem on $\bG / \bB$).  Here $w_0\in \Sigma$ is the longest element, $T$ is the pull-back of the tangent bundle under $\cR_{\leq w_0} \rightarrow \overline{\cF^{w_0}} \cong \bG / \bB$, and $\Delta \in H_\bT^\ast(pt)$ is the product of the positive roots of $\bG$. Since $\idemCB$ is a full idempotent of $\nilHecke$, it is also a full idempotent of $\ICB$.

\subsection{Relationship between \texorpdfstring{$\CB$ and $\ICB$}{Asph or A}}
\label{sec: CB and ICB}

By \cite[Theorem 3.3]{Web3} there is an isomorphism $\CB \cong \idemCB \ICB \idemCB$, and in fact $\ICB$ is a matrix algebra over $\CB$ of rank $|\Sigma|$ (c.f.~also \cite[Section 4.2]{BEF}, \cite[Corollary 3.8]{FT1}).  Our goal in this section is to give a slightly different perspective on this result, which is more explicit, inspired by results from \cite{Sauter}. To this end, consider
\begin{equation}
\bimIK \DEF H_\ast^{(\Iwa \times \bF)\rtimes \CC^\times}( \cR^{\mathsf{sph}} ), \qquad \bimKI \DEF H_\ast^{(\GO \times \bF)\rtimes \CC^\times}(\cR)
\end{equation}
Using straightforward variations on the convolution diagram (3.2) in \cite{BFN1}, and on the proof of associativity \cite[Theorem 3.10]{BFN1}, we have:

\begin{Lemma}
There are convolution products
\begin{align*}
\ICB \otimes \bimIK \longrightarrow \bimIK,\qquad  &\bimIK \otimes \CB \longrightarrow \bimIK, \\
\bimKI \otimes \ICB \longrightarrow \bimKI, \qquad & \CB \otimes \bimKI \longrightarrow \bimKI, \\
\bimIK \otimes \bimKI \rightarrow \ICB, \qquad &\bimKI \otimes \bimIK \rightarrow \CB \label{eq: convolving bimodules}
\end{align*}
All compositions of convolution products are associative, whenever defined.
\end{Lemma}

We denote these convolution products by $\ast$.  In particular, convolution makes $\bimKI$ an $\CB$--$\ICB$--bimodule and $\bimIK$ an $\ICB$--$\CB$--bimodule, and defines bimodule homomorphisms $\bimKI\otimes_{\ICB} \bimIK \longrightarrow \CB$ and $\bimIK \otimes_{\CB} \bimKI \longrightarrow \ICB$.  
Now, consider the elements 
\begin{equation}
\Inc \DEF [\cR^{\mathsf{sph}}_{\leq 1}] \in \bimIK, \qquad \Av \DEF \frac{1}{|\Sigma|} \operatorname{eu}(T) \cap [\cR_{\leq w_0}] \in \bimKI
\end{equation}
where the latter is defined as in (\ref{eq: geometric idempotent}) (note that this is a $\GO$--equivariant cycle).  We will need some basic properties of these elements under the above convolution products.

\begin{Lemma}
\label{lemma: Av and Inc}
\begin{enumerate}
\item[(a)] For $a\in \CB$ (or $a\in \bimKI$) the convolution $\Inc \ast a = \operatorname{Forg}(a)$, where $\operatorname{Forg}$ is the natural homomorphism of restriction from $\GO$-- to $\Iwa$--equivariance.
\item[(b)] For $b \in \ICB$ (or $b\in \bimIK$) the convolution $b \ast \Inc = \pi_\ast( b)$, where $\pi : \cR \rightarrow \cR^{\mathsf{sph}}$ is the natural map $[g, s] \mapsto [g,s]$.
\item[(c)] $\Inc \ast \Av = \idemCB \in \ICB$.
\item[(d)] $\Av \ast \Inc = \unit \in \CB$ is the unit element.
\end{enumerate}
\end{Lemma}
\begin{proof}
Parts (a), (b) are similar to the proof that $\unit$ is the unit element in \cite[Theorem 3.10]{BFN1}. ee also parts (a), (b) of \cite[Lemma 5.7]{BFN1}.  Parts (c) and (d) follow immediately from parts (a) and (b), respectively.
\end{proof}

We now come to the main result of this section:
\begin{Theorem}
\label{thm: ICB matrix algebra over CB}

\begin{enumerate}
\item[(a)] The map $\CB \rightarrow \ICB$, $a \mapsto \Inc \ast a \ast \Av$ gives an isomorphism of algebras $\CB \cong \idemCB \ICB \idemCB$.  It sends $\unit \mapsto \idemCB$.

\item[(b)] The map $\bimIK \rightarrow \ICB$, $a\mapsto a\ast \Av$ identifies $\bimIK \cong \ICB \idemCB$, compatibly with bimodule structures under (a). It sends $\Inc \mapsto \idemCB$.

\item[(c)] The map $\bimKI \rightarrow \ICB$, $a\mapsto \Inc \ast a$ identifies $\bimKI \cong \idemCB \ICB$, compatibly with bimodule structures under (a).  It sends $\Av \mapsto \idemCB$.
\end{enumerate}
 Moreover $\ICB$ is a matrix algebra over $\CB$ of rank $|\Sigma|$, and is generated by its subalgebras $\nilHecke$ and $\CB$.
\end{Theorem}
\begin{proof}
We prove part (a), with (b), (c) being similar. Using Lemma \ref{lemma: Av and Inc}, it is easy to see that our map $\CB \rightarrow \ICB$ lands in $\idemCB \ICB \idemCB$, and that its inverse $\idemCB \ICB \idemCB \rightarrow \CB$ is given by $b\mapsto \Av \ast b \ast \Inc$.  Therefore it is an isomorphism.

Finally, since $\nilHecke$ is a matrix algebra, with a complete set of minimal idempotents (all isomorphic to $\idemCB$), we get the final claim of the theorem.
\end{proof}

As a consequence, since $\idemCB\in \ICB$ is a full idempotent, we see that $\CB$ is Morita equivalent to $\ICB$.  This Morita equivalence is witnessed by the bimodules $\bimIK, \bimKI$. 

\begin{Remark}
As mentioned above, this result is a variation on \cite[Theorem 3.3]{Web3}.  Modulo addressing some issues of infinite-dimensionality, our two approaches should agree as a consequence of \cite{Sauter}.  
\end{Remark}

\begin{Remark}
Because of the embedding $\CC[\Sigma] \subset \nilHecke\subset \ICB$,  we get actions of $\Sigma$ on $\ICB$ by left and right multiplication.  Since $\idemCB \in \ICB$ is the symmetrizing idempotent for $\Sigma$, we can rephrase the theorem: the left $\Sigma$--invariants of $\ICB$ are isomorphic to $\bimKI$, the right $\Sigma$--invariants to $\bimIK$, and the $\Sigma\times\Sigma$-invariants are isomorphic to $\CB$.  This is analogous to \cite[Proposition 1]{Sauter}.
\end{Remark}

An essential property of the above isomorphism $\CB \cong \idemCB \ICB \idemCB$ is the following:

\begin{Lemma}
\label{lemma: embedding useful fact}
Let $a \in \CB$ and $b\in \ICB$ be such that $\operatorname{Forg}(a) = \pi_\ast(b)$.  Then
$$
\Inc \ast a \ast \Av = b \ast \idemCB
$$
\end{Lemma}
\begin{proof}
$\Inc \ast a = \operatorname{Forg}(a)$ by Lemma \ref{lemma: Av and Inc}(a), which is an element of $\bimIK$. Meanwhile, $b\ast \idemCB \in \ICB \idemCB$. It suffices to verify that these elements correspond to one another under the isomorphism $\bimIK \cong \ICB \idemCB$ from part (b) of the previous theorem.  Note that the inverse to this isomorphism is the map $g\mapsto g\ast \Inc$.  Applying this inverse map to $b\ast \idemCB$, by Lemma \ref{lemma: Av and Inc} we get
$$
(b \ast  \idemCB)\ast \Inc = b \ast \Inc = \pi_\ast(b) = \operatorname{Forg}(a),
$$
which proves the claim.
\end{proof}

\subsection{Generation by minuscules} 
\label{sec: generation by minuscules general case}
  
Let $\laG$ be a minuscule coweight for $\bG$.  Then there is an isomorphism $\bG / \bP \cong  \Gr^\laG = \overline{\Gr^\laG}$ defined by $g \bP \mapsto g z^\lambda$, where $\bP = \bP_\laG \subset \bG$ is a parabolic subgroup (the stabilizer of $z^\laG$ in $\bG$).  Note that since $\laG$ is dominant $\bB^- \subset \bP$.  The associated parabolic Weyl group $\Sigma_\laG \subset\Sigma$ is the stabilizer of $\laG$, and we will denote its longest element by $w_\laG$.  

Recall from Section \ref{section: monopole} that there are elements $M_{\laG, f}   \in \CB$, for $f \in P^{\Sigma_\laG}$.
The following result generalizes \cite[Lemma 4.5]{KTWWY2}:
\begin{Lemma}
\label{lemma: minuscule monopole reexpressed}
Suppose that $\laG$ is a minuscule coweight.  Under $\CB \stackrel{\sim}{\rightarrow} \idemCB \ICB \idemCB$, for any $f\in P^{\Sigma_\laG}$ we have
$$
M_{\laG, f} \mapsto \partial_{w_0 w_\laG} f r_\laG \cdot \idemCB
$$
\end{Lemma}

\begin{proof}
Denote $w = w_0 w_\laG$ for brevity. By Lemma \ref{lemma: embedding useful fact}, it suffices to prove that $\pi_\ast(\partial_w f r_\lambda) = \operatorname{Forg}(M_{\laG, f})$.  This reduces to a calculation for finite dimensional (partial) flag varieties.  Consider the natural proper map $\pi: \bG / \bB^- \rightarrow \bG / \bP$.  Its restriction to the Schubert variety $\pi: X^w \rightarrow \bG / \bP$ is a proper birational map, since the element $w$ is the minimal length coset representative of $w_0 \Sigma_\laG$ (by e.g.~\cite[Corollary 7.1.18]{Kumar}).  As in Section \ref{sec: nilhecke algebra} we can think of $\partial_w = [ X^w]$, so we conclude that $\pi_\ast (\partial_w) = [\bG / \bP]$. For any $g \in H_\bG^\ast(\bG/\bP) \cong H_\bT^\ast(pt)^{W_\laG}$, 
$$
g \cap [ \bG / \bP ]  =  g \cap \pi_\ast \partial_w = \pi_\ast( \pi^\ast(g) \cap \partial_w)
$$
by the projection formula. By the localization theorem at $\bT$--fixed points on $\bG /\bB^-$, we can write $\partial_w = \sum_{v \leq w} a_v v$ for some $a_v \in \operatorname{Frac} H_\bT^\ast(pt)$.\footnote{Since $\pi_\ast (\partial_w) = [\bG/ \bP]$, these coefficients have the property that, for any fixed coset $x W_\laG$,  the sum  
$$\sum_{\substack{v\leq w, \\ v W_\laG = x W_\laG}} a_v  = \frac{1}{\operatorname{eu}(T_x \bG/ \bP)} $$.}
We find that, as elements of $\nilHecke$,
$$
\pi^\ast(g) \cap \partial_w = \sum_{v \leq w} v(g) a_v  v  = \sum_{v\leq w} a_v v \cdot g = \partial_w g
$$
Applying this discussion to $g = f \cdot \operatorname{eu}( z^\laG \NO / z^\laG \NO \cap \NO)$ proves the claim, by \cite[Proposition A.2]{BFN2}
\end{proof}

We say that $\CB$ is {\bf generated by minuscules} if it is generated as an algebra by $H_{\bG\times\bF\times\CC^\times}^\ast(pt)$ and the dressed minuscule monopole operators $M_{\laG, f}$.

\begin{Corollary}
\label{lem: generated by minuscules, generators of Iwahori}
If $\CB$ is generated by minuscules, then $\ICB$ is generated by its subalgebras $\nilHecke$ and $\ACB$.
\end{Corollary}
\begin{proof}
Consider the subalgebra $\ICB' \subseteq \ICB$ generated by $\nilHecke$ and $\ACB$.  To prove that $\ICB' = \ICB$, by the final part of Theorem \ref{thm: ICB matrix algebra over CB} it suffices to show that the images of generators of $\CB$ land in $\idemCB \ICB' \idemCB $ under the isomorphism $\CB \cong \idemCB \ICB \idemCB$.  Lemma \ref{lemma: minuscule monopole reexpressed} shows that this is true for the elements $M_{\laG, f}$.  Since we assume $\CB$ is generated these elements, this proves the claim.
\end{proof}

There is a sort of converse to Lemma \ref{lemma: minuscule monopole reexpressed}, which we will need later: for $\muG$ in the Weyl orbit of a dominant minuscule coweight $\laG$, we can express $r_\muG \in \ACB\subset \ICB$ using only $\nilHecke$ and the dressed monopole operators $M_{\laG, f}$.  To prove this, we will use the following basic fact about nilHecke algebras:

\begin{Lemma}
\label{lemma: dual bases}
There exist bases $\{x_w\}_{w \in \Sigma}$ and $\{y_{w}\}_{w\in \Sigma}$ for $H_\bT^\ast(pt)$ over $H_\bG^\ast(pt)$, such that the following identity holds in $\nilHecke$:
$$
1 = \sum_{w \in  \Sigma} x_{w} \partial_{w_0} y_{w} 
$$
\end{Lemma}
In type A, one can define such bases using Schubert polynomials, see \cite[Remark 2.5.2]{KLMS}.
\begin{proof}
Consider the bilinear pairing $H_\bT^\ast(pt) \otimes_{H_\bG^\ast(pt)} H_\bT^\ast(pt) \longrightarrow H_\bG^\ast(pt)$ defined by $f\otimes g \mapsto \partial_{w_0} (fg)$.  Our claim is equivalent to the non-degeneracy of this pairing, i.e.~that dual bases exist so that $\partial_{w_0}(x_w y_v) = \delta_{w,v}$.  This pairing can be identified with the $\bG$--equivariant Poincar\'e pairing for $\bG / \bB$, see e.g.~\cite[Section 3]{HolmSja}. Therefore, non-degeneracy follows from equivariant Poincar\'e duality \cite[Proposition 1]{Brion2}.
\end{proof}

\begin{Proposition}
\label{prop: abelian as monopole}
Let $\laG$ be dominant minuscule, and $\muG \in \Sigma \laG$.  Then there exist elements $x_p, y_p \in H_{\bT\times \CC^\times}^\ast(pt)$ and $f_p \in H_{\bT\times \bF\times\CC^\times}^\ast(pt)$ for $p$ running in some finite index set, such that
$$
r_\muG = \sum_p  x_p ( \Inc \ast M_{\laG, f_p} \ast \Av) \partial_{w_0} y_p
$$
as elements of $\ICB$.
\end{Proposition}
\begin{proof}
In Section \ref{section: ACB and ICB} we used the embedding $\cR^{\mathsf{ab}}\subset \cR$, but we note that there is a similar embedding $\cR^{\mathsf{ab}} \subset \cR^{\mathsf{sph}}$.  Under the latter $\cR^{\mathsf{ab}}_\muG \subset \cR^{\mathsf{sph}}_{\leq \laG}$, so there is a corresponding fundamental class
$$
r_\muG^{\mathsf{sph}} \in H_\ast^{(\TO\times\bF)\rtimes\CC^\times}(\cR^{\mathsf{sph}}_{\leq \laG}) \cong H_{\bT}^\ast(pt) \otimes_{H_\bG^\ast(pt)} H_\ast^{(\GO\times \bF)\rtimes \CC^\times}(\cR^{\mathsf{sph}}_{\leq \laG} ) 
$$
where the isomorphism is \cite[Lemma 5.3]{BFN1}.  Therefore there exist elements $g_i \in H_\bT^\ast(pt)$ and $f_i \in H_{\bT^\times \bF\times \CC^\times}^\ast(pt)^{\Sigma_\laG}$ such that $r^{\mathsf{sph}}_\muG = \sum_i g_i \operatorname{Forg}( M_{\laG, f_i})$.  It is clear that the restriction of $\pi : \cR \rightarrow \cR^{\mathsf{sph}}$ to $\cR^{\mathsf{ab}}_\muG$ is an isomorphism, so $\pi_\ast (r_\muG) = r_\muG^{\mathsf{sph}}$.  We can think of $r_\muG^{\mathsf{sph}}\in \bimIK$, and as in Lemma \ref{lemma: embedding useful fact} we deduce that in $\ICB$
$$
r_\muG \ast \idemCB = r_\mu^{\mathsf{sph}} \ast \Av = \sum_i g_i ( \Inc \ast M_{\laG, f_i} \ast \Av)
$$
Now, by \cite[Lemma 3.21]{BFN1}, for any $x\in H_{\bT\times\CC^\times}^\ast(pt)$ we have $r_\muG x = \mathsf{u}_\muG(x) r_\muG$, where $\mathsf{u}_\muG$ is an $\hbar$--difference operator.  Also, in $\nilHecke$ we note the basic relation $\idemCB \partial_{w_0} = \partial_{w_0}$.  Putting this all together, we have
$$
\sum_{w, p} \mathsf{u}_\muG(x_w) g_i ( \Inc \ast M_{\laG, f_i} \ast \Av) \partial_{w_0} y_w  = \sum_w \mathsf{u}_\muG(x_w) r_\muG \idemCB \partial_{w_0} y_w = \sum_w r_\muG  x_w \partial_{w_0} y_w  = r_\muG
$$
where the final equality is Lemma \ref{lemma: dual bases}. The left-hand side is of the claimed form, so we are done.
\end{proof}

\section{Generators in the quiver gauge theory case}
\label{section: generators in the quiver gauge theory case}

In this section we focus on the quiver gauge theory case, and apply the results developed in the previous section.  

\subsection{Generation by minuscules}
Let us fix some notation. For each $i\in I$ and $1\leq r \leq \bv_i$, denote by $\eps_{i,r}$ the coweight $(0,\ldots, 1,\ldots, 0)$ of $\GL(V_i)$ (with 1 in its $r$th component).  Following the notation of \cite[Appendix A(ii)]{BFN2}, recall that the minuscule coweights of $\GL(V_i)$ are 
\begin{align*}
\varpi_{i, n} &= \sum_{r=1}^n \eps_{i,r} = (1,\ldots, 1, 0,\ldots, 0), \\
\varpi_{i,n}^\ast & = - \sum_{r=1}^n\eps_{i, n-r+1} = (0,\ldots, 0, -1, \ldots, -1)
\end{align*}
for $1\leq n \leq \bv_i$. The minuscule coweights for the group $\bG = \prod_{i\in I} \GL(V_i)$ correspond to a choice of a minuscule coweight (or zero) of $\GL(V_i)$  at each vertex $i \in I$.

The next result is known to experts \cite[Remark 6.7]{BFN1}, \cite[Section 4.3]{BDG}.  It is proven during the course of the proof of \cite[Theorem 4.29]{FT2}, but for the convenience of the reader we include a proof.

\begin{Proposition}
\label{prop: quiver generated by minuscules}
For any quiver gauge theory $\CB$ is generated by minuscules. 
\end{Proposition}

\begin{proof}
We will apply \cite[Proposition 6.8]{BFN1}, whose proof provides a procedure to find a generating set for $\CB$.
Consider the dominant Weyl chamber for $G = \prod_i GL(m_i)$, and inside it the hyperplane arrangement given by the weights of $\bN$.  For each chamber of this arrangement, we find generators for its semigroup of integral points, and for each such generator $\laG$, we consider the monopole operators $M_{\laG, f}$ (more precisely, we must choose some lifts from the associated graded, see Remark \ref{remark: monopoles not welldef}).  Taking all of these elements, together with $H_{\bG\times \bF\times\CC^\times}^\ast(pt)$ we get a generating set for $\CB$. 

Note that if we refine the hyperplane arrangement by adding more hyperplanes, the above procedure applied to the refined arrangement will still produce a set of generators for $\CB$ (possibly with redundancies).  In particular, for a quiver gauge theory observe that we may refine the arrangement coming from $\bN$ to include the hyperplanes $\eV_{j, s} - \eV_{i,r} = 0$ and $\eV_{i,r} = 0$ for {all} $i, j, r, s$ (i.e.~even if $i\not\sim j$).   Denote a point $\eta = (\eta_{i,r}) \in \ft_{\RR}$.  Then each chamber of the refined arrangement corresponds to an ordering (as real numbers)
\begin{equation*}
\eta_{i_1, r_1} > \eta_{i_2, r_2} > \ldots > \eta_{i_p, r_p} > 0 > \eta_{i_{p+1}, r_{p+1} } > \ldots > \eta_{i_{|\bv|}, r_{|\bv|}}
\end{equation*}
where we denote $ |\bv| = \sum_i \bv_i$.  This chamber roughly looks like a Weyl chamber for $\GL(|\bv|)$; it is easy to see that its semigroup of integral points is generated by the set
$$
\left\{\sum_{a=1}^s \eps_{i_a, r_a} : 1\leq s \leq p \right\} \cup \left\{ - \sum_{a=s}^{|\bv|} \eps_{i_a, r_a} : p+1 \leq s \leq |\bv| \right\}
$$
To complete the proof, we simply observe that every element of this set is a minuscule coweight for $\bG$.  

\end{proof}

As an immediate consequence of this proposition and Corollary \ref{lem: generated by minuscules, generators of Iwahori}, we have:
\begin{Corollary}
\label{cor: quiver ICB generated by ACB and nilHecke}
For any quiver gauge theory, $\ICB$ is generated by its subalgebras $\ACB$ and $\nilHecke$.
\end{Corollary}

\begin{Remark}
Proposition \ref{prop: quiver generated by minuscules} and Corollary \ref{cor: quiver ICB generated by ACB and nilHecke} give honest algebra generators, i.e.~no Poisson brackets or division by $\hbar$ are used.  In particular, both results remain true at $\hbar = 0$.
\end{Remark}

\begin{Remark}
Suppose that $\bG = \prod_i \GL(V_i)$, and that $\bN$ is a representation of $\bG$ such that its weights are all multiples of $\eV_{j,s} - \eV_{i,r}$ or $\eV_{i,r}$.  Then it is clear that Proposition \ref{prop: quiver generated by minuscules} and Corollary \ref{cor: quiver ICB generated by ACB and nilHecke} still apply for $(\bG, \bN)$.  But as far as the author is aware, the only $\bN$ with this property come from quivers (allowing loops and multiple edges).
\end{Remark}

\subsection{Generators for \texorpdfstring{$\ACBone$}{Aab at h=1}}

Let $\laG = (\laG_{i,r})$ be any coweight for $\bG$.  If $\laG_{i,r}\geq 0$ for some $i,r$, then applying \cite[Section 4(iii)]{BFN1} we find that in $\ACB$ we have
\begin{equation}
\label{eq: product formula}
r_{\eps_{i,r}} r_\laG =   \Bigg(\prod_{i\rightarrow j} \prod_{\substack{1\leq s \leq \bv_j \\ \laG_{i,r} < \laG_{j, s}} }(\eV_{j,s } - \eV_{i,r}- \tfrac{1}{2} \hbar) \Bigg) \Bigg(\prod_{i\leftarrow j} \prod_{\substack{1\leq s \leq \bv_j \\ \laG_{i,r} < \laG_{j, s}} }(\eV_{i,r} - \eV_{j,s} + \tfrac{1}{2} \hbar ) \Bigg) r_{\eps_{i,r}+ \laG} 
\end{equation}
There is a similar expression for the product $r_{-\eps_{i,r}} r_\laG $ when $\laG_{i,r}\leq 0$.  Notice that if there are no pairs $(j,s)$ such that $\laG_{i,r} < \laG_{j,s}$, then the above product is simply $r_{\eps_{i,r} + \laG}$.  The proof below is in some sense a generalization of this observation.

\begin{Proposition}
\label{prop: generators for abelian theory}
$\ACBone$ is generated by $H_{\bT\times \bF}^\ast(pt)$ and the elements $r_{\pm \eps_{i,r}}$.
\end{Proposition}
\begin{proof}
It is sufficient to prove that any $r_\laG$ can be expressed in terms of these elements.  We proceed by induction on $|\laG| := \sum_{i,r} |\laG_{i,r}|$.  When $|\laG| = 0$ the statement is obviously true, so take $\laG$ some arbitary coweight with $|\laG| >0$.  

Suppose that there exists some $\laG_{i,r} > 0$.  Then we will express $r_\laG$ as a linear combination of the products $r_{\eps_{i,r}} r_{\laG - \eps_{i,r}} $ where $\laG_{i,r} > 0$,  thought of as elements of $\ACBone$ viewed as a left module over $H_{\bT}^\ast(pt)$.  For $\laG_{i,r}>0$, by equation (\ref{eq: product formula}) we have
\begin{equation} 
\label{eq: product formula redux}
r_{\eps_{i,r}} r_{\laG - \eps_{i,r}} = \pm \Bigg(\prod_{i \sim j} \prod_{\substack{1\leq s\leq \bv_j \\ \laG_{i,r} \leq \laG_{j, s}}} (\eV_{i,r} - \eV_{j,s} + \tfrac{1}{2})   \Bigg) r_\laG 
\end{equation}
Consider the set $S$ of all pairs $(i, r)$ such that $\laG_{i,r}$ is maximal.  If $(i,r)\in S$, then the polynomial on the right-hand side of (\ref{eq: product formula redux}) contains only factors where $(j,s)\in S$.  Taken together over all $(i,r)\in S$, these polynomials therefore define an ideal $\mathcal{I}$ inside the subring
\begin{equation*}
\CC[ \eV_{i,r} : (i,r) \in S] \ \subseteq \ H_{\bT}^\ast(pt),
\end{equation*}

We claim that the vanishing locus  $\mathbb{V}(\mathcal{I})$ has no $\CC$--points.  Assuming this claim for the moment, it follows that $1 \in \mathcal{I}$.  Using (\ref{eq: product formula redux}), we can therefore express $r_\laG$ as a (left) linear combination over $\CC[ \eV_{i,r} : (i,r) \in S]$ of the elements $r_{\eps_{i,r}} r_{\laG - \eps_{i,r}}$, as desired.  Now, to prove that $\mathbb{V}(\mathcal{I})$ has no points, let us assume that $x = (x_{i,r}) \in \mathbb{V}(\mathcal{I})$ is a $\CC$--point.  For every $(i, r)\in S$, since the polynomial on the right-hand side of (\ref{eq: product formula redux}) must vanish, we must have
\begin{equation}
\label{eq: element condition}
x_{i,r} \in \left\{ x_{j,s} - \tfrac{1}{2} : (j,s) \in S, \ i \sim j, 1 \leq s \leq \bv_j\right\} 
\end{equation}
As $x$ is a finite list of complex numbers some $x_{i,r}$ has maximal real part, so (\ref{eq: element condition}) is impossible.  Hence $\mathbb{V}(\mathcal{I})$ is empty.

A similar argument applies to the case when there exists $\laG_{i,r}<0$. Putting these two cases together, we see that any $r_\laG$ is expressible over $H_T^\ast(pt)$ as a linear combination of summands $r_{\pm \eps_{i,r}} r_{\laG \mp \eps_{i,r}}$.  As $|\laG\mp\eps_{i,k}| < |\laG|$, by induction on $|\laG|$ the result follows.
\end{proof}

Because of Corollary \ref{cor: quiver ICB generated by ACB and nilHecke}, we can conclude that:
\begin{Corollary}
\label{cor: gens of Iwahori}
As an algebra over $H_{\bF}^\ast(pt)$, $\ICBone$ is generated by $\nilHecke$ and the elements $r_{\pm \eps_{i, r}}$.
\end{Corollary}

\subsection{Generators for \texorpdfstring{$\CBone$}{Asph at h=1}} 

We now come to a main result of this paper:

\begin{Theorem}
\label{thm: Yangian surjectivity}
$\CBone$ is generated by $H_{\bG\times\bF}^\ast(pt)$ and the dressed minuscule monopole operators $M_{\varpi_{i,1}, f}$, $M_{\varpi_{i,1}^\ast, f}$.
\end{Theorem}

\begin{proof}
Consider the subalgebra $\cS^{\textsf{sph}}\subset \CBone$ generated by these elements, and also the corresponding subalgebra $\cS\subset \ICBone$ generated by $\cS^{\textsf{sph}}$ and $\nilHecke$.  It is enough to prove that $\cS = \ICBone$.  Because of Corollary \ref{cor: gens of Iwahori}, it further suffices to show that $r_{\pm \eps_{i,r}} \in \cS$ for all $i,r$.  Since $\eps_{i, r} \in \Sigma \varpi_{i,1}$ and $- \eps_{i,r} \in \Sigma \varpi_{i,1}^\ast$, this follows from Proposition \ref{prop: abelian as monopole}.
\end{proof}
	
We note that all of our $\hbar = 1$ results above are also true over $\CC[\hbar, \hbar^{-1}]$, with slightly modified proofs. In general they will fail over $\CC[\hbar]$, where one must give more generators.  However: 

\begin{Conjecture}
\label{conj: poisson gens}
Consider the Poisson bracket $\{a, b\} = \tfrac{1}{\hbar}[a,b]$ on $\CB$.  Then $\CB$ is Poisson generated by $H_{\bG\times\bF\times \CC^\times}^\ast(pt)$ and the dressed minuscule monopole operators $M_{\varpi_{i,1}, f}$, $M_{\varpi_{i,1}^\ast, f}$
\end{Conjecture}

\begin{Remark}
\label{rmk: only need f=1}
In above the theorem and conjecture, it suffices to take dressing $f=1$.  Indeed, as we discuss below, the above elements satisfy relations coming from the Yangian. In particular one can write down elements $S_i^{(r)} \in H_{\bG}^\ast(pt) \subset \CB$, as in \cite[Remark 3.10]{FKPRW}. These elements $S_i^{(r)}$ can be used to generate all $M_{\varpi_{i,1}, f}$ from $M_{\varpi_{i,1}, 1}$ under Poisson bracket (and similarly for $M_{\varpi_{i,1}^\ast, f}$).
\end{Remark}

\subsection{Relationship to (truncated) shifted Yangians}
\label{section: relationship to shifted Yangians}

Suppose that our quiver $Q$ is an orientation of a Dynkin diagram of finite ADE type, with corresponding simple Lie algebra $\fg_Q$.  Given $I$--graded vector spaces $W, V$, we can define a pair of coweights $\lambda, \mu$ for $\fg_Q$, as in \cite[Section 3(iii)]{BFN2}. By \cite[Theorem B.18]{BFN2} specialized at $\hbar=1$, Theorem \ref{thm: Yangian surjectivity} translates directly into the statement that there is a surjection 
\begin{equation}
\label{eq: Yangian surjection}
Y_\mu[z_1,\ldots,z_N] \twoheadlongrightarrow \CBone
\end{equation}
Recall from Section \ref{notation: cohomology} that we identify $H_\bF^\ast(pt) = \CC[z_1,\ldots,z_N]$.  Here, $Y_\mu[z_1,\ldots,z_N]$ is a {\bf shifted Yangian for} $\fg_Q$.
We will not recall its precise definition, and refer the reader to \cite[Definition B.2]{BFN2} or \cite[Definition 4.1]{KTWWY2}. For us, it suffices to know that its generators map as follows:
\begin{align}
A_i^{(r)} & \mapsto (-1)^r e_r(\eV_{i,1},\ldots,\eV_{i,\bv_i}) \label{eq: Y to CB 1} \\
E_i^{(1)} & \mapsto (-1)^{\bv_i} M_{\varpi_{i,1}^\ast, 1} \label{eq: Y to CB 2} \\
F_i^{(1)} & \mapsto (-1)^{\sum_{i\rightarrow j} \bv_j} M_{\varpi_{i,1}, 1} \label{eq: Y to CB 3}
\end{align}
where $e_r$ denotes the $r$th elementary symmetric polynomial.  Note that (\ref{eq: Y to CB 1}--\ref{eq: Y to CB 3}) uniquely determines the images of the higher degree generators $E_i^{(r)}, F_i^{(r)}$, as in Remark \ref{rmk: only need f=1}. 

As in \cite[Definition 4.1]{KTWWY2}, one can appropriately define $Y_\mu[z_1,\ldots,z_N]$ in arbitrary simply-laced Kac-Moody type.  To be more precise, in this paper we follow \cite[Definition 4.1]{KTWWY2} except:
\begin{enumerate}
\item[(a)] We take the roots of $p_i(u)$ be the generators of the polynomial ring $H_{\bF_i}^\ast(pt)$, so $ p_i(u) = \prod_{k: i_k = i} (u-z_{k}) $.
\item[(b)] We specialize ``$\hbar =1$'' instead of ``$\hbar=2$'', meaning we divide the right-hand sides of all relations by two, and let
$$
H_i(u) = p_i(u) \frac{\prod_{j\sim i} (u-\tfrac{1}{2})^{\bv_j} }{u^{\bv_i} (u-1)^{\bv_i}} \frac{\prod_{j\sim i} A_j(u-\tfrac{1}{2}) }{A_i(u) A_i(u-1)}
$$  
\end{enumerate}
There is then a map (\ref{eq: Yangian surjection}) for any simple quiver $Q$, defined by the same formulas (\ref{eq: Y to CB 1}--\ref{eq: Y to CB 3}).  The fact that the relations for the shifted Yangian of $\fg_Q$ are satisfied is ultimately a rank 2 calculation, as in \cite[Appendix B]{BFN2}. The following is now an immediate consquence of Theorem \ref{thm: Yangian surjectivity}:

\begin{Corollary}
\label{cor: h=1 version}
For any simple quiver $Q$, there is a surjective homomorphism of $\CC[z_1,\ldots,z_N]$--algebras
$$
Y_\mu[z_1,\ldots,z_N] \twoheadlongrightarrow \CBone,
$$
defined by (\ref{eq: Y to CB 1}--\ref{eq: Y to CB 3})
\end{Corollary}

\begin{Remark}
As explained in \cite[Remark 4.4]{KTWWY2}, the algebra $Y_\mu[z_1,\ldots,z_N]$ is not quite built out of the definition of the Yangian (in Kac-Moody type) from \cite[Section 2]{GNW}.  Rather, it is defined using the non-standard Cartan generators $A_i^{(r)}$, which better fit into the Coulomb branch picture (because of equation (\ref{eq: Y to CB 1})).
\end{Remark}

\begin{Remark}
\label{remark: Zastava}
As in \cite[Remark 3.15]{BFN2}, we can define $\cR^{\mathsf{sph}}_+ \subset \cR^{\mathsf{sph}}$ as the preimage of the positive part $\Gr^+ \subset \Gr$ of the affine Grassmannian.  We may also consider the Borel Yangian $Y_\mu^{\leq}[z_1,\ldots,z_N] \subset Y_\mu[z_1,\ldots, z_N]$, which is the subalgebra generated by all $F_i^{(r)}$, $A_i^{(r)}$ and $z_k$.  Tracing through the proof, we can see that the above surjection restricts to a surjection on subalgebras
$$
Y_\mu^{\leq}[z_1,\ldots,z_N] \twoheadlongrightarrow H_\ast^{(\GO\times \bF)\rtimes \CC^\times} (\cR^{\mathsf{sph}}_+)_{\hbar =1}
$$
and similarly for the opposite Borel Yangian and $\cR^{\mathsf{sph}}_{-}$.  The right-hand side  quantizes the Zastava space $Z^{\lambda - \mu}$, at least in finite ADE type (but see \cite[Section 3(vi)]{BFN2}). The above relation to Borel Yangians generalizes \cite{FR1}, \cite{FR2}.
\end{Remark}

Corollary \ref{cor: h=1 version} may be rephrased as an isomorphism $Y_\mu^\lambda \stackrel{\sim}{\rightarrow} \CBone$, where $Y_\mu^\lambda$ is the {\bf truncated shifted Yangian} defined in \cite[Section 4.2]{KTWWY2} and \cite[Appendix B(viii)]{BFN2}. $Y_\mu^\lambda$ can be defined as the image of $Y_\mu[z_1,\ldots,z_N] \rightarrow \CBone$.   (As for the shifted Yangian, in this paper we think of $Y_\mu^\lambda$ as an algebra over $\CC[z_1,\ldots,z_N]$.)

\subsubsection{Flag Yangians}
Following \cite[Section 4.3]{KTWWY2}, we can also define the {\bf flag Yangian} $FY_\mu^\lambda$, which is a matrix algebra over $Y_\mu^\lambda$ built using $\nilHecke$.  Using Theorem \ref{thm: ICB matrix algebra over CB}, it is easy to see that the above isomorphism extends  to an isomorphism $FY_\mu^\lambda \stackrel{\sim}{\rightarrow} \ICBone$, which is compatible with the isomorphism of their spherical subalgebras $Y_\mu^\lambda$ and $\CBone$. Put differently, the embedding $Y_\mu^\lambda \hookrightarrow FY_\mu^\lambda$ is defined in terms of their faithful actions on the polynomial rings 
$$
P^\Sigma = H_{\bG \times \bF}^\ast(pt) \subset P = H_{\bT\times \bF}^\ast(pt)
$$
Elements of $Y_\mu^\lambda$ act on $P^\Sigma$, and we extend this to an action on $P$ by projection/inclusion from $P^\Sigma$.  This agrees with $\CBone \hookrightarrow \ICBone$:  the element $\Av\in \ICBone$ corresponds to projection onto $\Sigma$--invariants, while $\Inc\in \ICBone$ corresponds to their inclusion.

\subsubsection{Filtrations in finite type}

In finite ADE type the algebra $Y_\mu$ has a family of filtrations $F_{\nu_1,\nu_2} Y_\mu$, parametrized by coweights satisfying $\nu_1+\nu_2 = \mu$.  These filtrations are defined explicitly in terms of PBW generators  \cite[Section 5.4]{FKPRW}, and extend to filtrations on $Y_\mu[z_1,\ldots,z_N]$ by declaring all $z_i$ to have degree 1.  The Rees algebras $\mathbf{Y}_\mu[z_1,\ldots,z_N] := \operatorname{Rees}^{F_{\nu_1,\nu_2}} Y_\mu [z_1,\ldots,z_N]$ are all canonically isomorphic as $\CC[\hbar,z_1,\ldots,z_N]$--algebras by \cite[Section 5.8]{FKPRW}, although they are not isomorphic as graded algebras.

Consider the filtration on $\CBone$ induced by half of the homological grading on $\CB$.  By choosing $\nu_1,\nu_2$ appropriately the map (\ref{eq: Yangian surjection}) respects filtrations, and thus there is an induced map of graded $\CC[\hbar,z_1,\ldots,z_N]$--algebras
\begin{equation}
\label{eq: Y to CB}
\bY_\mu [z_1,\ldots,z_N] \longrightarrow \CB
\end{equation}
 See \cite[Section B(vii)]{BFN2} for more details. It is known that the corresponding map of associated graded algebras is surjective:
\begin{equation*}
\gr Y_\mu[z_1,\ldots,z_N] \twoheadlongrightarrow \gr \CBone
\end{equation*}
by \cite[Lemma B.27]{BFN2}.  This corresponds to a closed embedding of schemes $\overline{\cW}_\mu^{\underline{\lambda}} \subset \cW_\mu \times \mathbb{A}^N$.  Our goal is to prove the following upgrade of this result:
\begin{Theorem}
\label{thm: finite type surjectivity}
For any $\mu$, there is a surjection of graded $\CC[\hbar]$--algebras
$$\mathbf{Y}_\mu[z_1,\ldots,z_N] \twoheadlongrightarrow \CB$$
\end{Theorem}

When $\mu$ is dominant this theorem was proven in \cite[Corollary B.28]{BFN2}, a result that we will make use of in our proof.  As a consequence of the theorem, Conjecture \ref{conj: poisson gens} is true in the finite ADE case, since the Yangian is Poisson generated by $\{ A_i^{(r)}, E_i^{(1)},F_i^{(1)}: r\geq 1, i \in I\}$, c.f.~Remark \ref{rmk: only need f=1}.

Let $\eta$ be a dominant coweight such that $\mu + \eta$ is dominant. By \cite[Proposition 3.8]{FKPRW} there are {\bf shift homomorphisms} $\iota_{\mu+\eta, -\eta, 0}$ and  $\iota_{\mu+\eta, 0, -\eta}$, which are both injective maps $Y_{\mu+\eta} \hooklongrightarrow Y_\mu$.  By looking at generators,  these maps {strictly} respect filtrations: $\iota_{\mu+\eta,-\eta,0}$ takes $F_{\nu_1+\eta, \nu_2} Y_{\mu+\eta} $ to $F_{\nu_1,\nu_2} Y_\mu$, while $\iota_{\mu+\eta, 0, -\eta}$ takes $F_{\nu_1, \nu_2+\eta} Y_{\mu+\eta}$ to $F_{\nu_1,\nu_2} Y_\mu$. Thus there are induced injective homomorphisms of $\CC[\hbar]$--algebras $\mathbf{Y}_{\mu+\eta} \hooklongrightarrow \mathbf{Y}_\mu$, which respect appropriate gradings.  We extend these homomorphisms $\CC[z_1,\ldots,z_N]$--linearly to get maps
\begin{equation}
\bY_{\mu+\eta}[z_1,\ldots,z_N] \hooklongrightarrow  \bY_\mu[z_1,\ldots,z_N]
\end{equation}

We can define analogous homomorphisms for Coulomb branches: choose a vector space $U = \bigoplus_{i\in I} U_i$ such that $\dim_\CC U_i = \eta_i$.  We consider $U$ to have trivial action of $\bF$.  As in \cite[Remark 5.14]{BFN1}, there are homomorphisms
\begin{align}
H_\ast^{(\GO \times \bF)\rtimes \CC^\times} ( \cR^{\mathsf{sph}}_{\bG, \bN \oplus \bigoplus_{i\in I} \Hom(U_i, V_i) } ) & \longrightarrow H_\ast^{(\GO\times \bF)\rtimes \CC^\times}( \cR^{\mathsf{sph}}_{\bG, \bN} ), \label{eq: CB shift 1} \\
H_\ast^{(\GO \times \bF)\rtimes \CC^\times} ( \cR^{\mathsf{sph}}_{\bG, \bN \oplus \bigoplus_{i\in I} \Hom(V_i, U_i) } ) & \longrightarrow H_\ast^{(\GO\times \bF)\rtimes \CC^\times}( \cR^{\mathsf{sph}}_{\bG, \bN} )  \label{eq: CB shift 2} 
\end{align}
See also \cite[Remark 3.11]{BFN2}.  Of course, the right-hand sides are both just $\CB$.  Meanwhile, the left-hand sides are also Coulomb branches of quiver gauge theories, {but with only partial flavour symmetry deformation}.  In fact the left-hand sides are isomorphic, since orientation of edges for a quiver gauge theory does not change the Coulomb branch up to isomorphism \cite[Section 6(viii)]{BFN1}. 

\begin{Lemma}
\label{lem: shift homomorphisms are compatible}
There is a commutative diagram of $\CC[\hbar,z_1,\ldots,z_N]$--algebras
$$
\begin{tikzcd}
\bY_{\mu+\eta}[z_1,\ldots,z_N]  \ar[r, "(\ref{eq: Y to CB})"] \ar[d, "\iota_{\mu+\eta, -\eta, 0}"] & H_\ast^{(\GO \times \bF)\rtimes \CC^\times} ( \cR^{\mathsf{sph}}_{\bG, \bN \oplus \bigoplus_{i\in I} \Hom(U_i, V_i) } ) \ar[d, "(\ref{eq: CB shift 1})"] \\
\bY_\mu[z_1,\ldots,z_N] \ar[r,"(\ref{eq: Y to CB})"] & \CB
\end{tikzcd}
$$
and similarly one relating $\iota_{\mu+\eta, 0, -\eta}$ and the homomorphism $(\ref{eq: CB shift 2})$.
\end{Lemma}

\begin{proof}
As mentioned above, the top-right corner of the diagram is the Coulomb branch of a quiver gauge theory.  The top arrow thus exists by our discussion above; note that we are considering only part of the possible flavour symmetry, ignoring the maximal torus of $\prod_i \GL(U_i)$.

To verify the commutativity of the diagram, we can embed the right-hand sides into a ring of difference operators via their respective homomorphisms $\bz^\ast (\iota_\ast)^{-1}$ from \cite[Appendix A(i)]{BFN2}.  By \cite[Lemma 5.24]{BFN1}, these homomorphisms $\bz^\ast (\iota_\ast)^{-1}$ are intertwined by (\ref{eq: CB shift 1}).  Now, it remains to verify that the images of  both Yangians' generators agree under the two composed maps to difference operators.  This reduces to comparing the images of the classes $M_{\varpi_{i,1}, f}$ and $M_{\varpi_{i,1}^\ast, f}$ in difference operators, which can be computed as in \cite[Appendix B]{BFN2}.  
\end{proof}

\begin{proof}[Proof of Theorem \ref{thm: finite type surjectivity}]
We make use of the commutative diagrams from Lemma \ref{lem: shift homomorphisms are compatible}.  First note that the top arrows are both surjections since $\mu+\eta$ is dominant, by \cite[Corollary B.28]{BFN2}.  It thus suffices to show that $\CB$ is generated by the combined images of the maps (\ref{eq: CB shift 1}) and (\ref{eq: CB shift 2}). 

To prove the latter claim, we apply a variation on Proposition \ref{prop: quiver generated by minuscules}: $\CB$ is generated by its subalgebras corresponding to the positive and negative parts of the affine Grassmannian (c.f.~Remark \ref{remark: Zastava}). The argument is similar, but simpler: the generators for any chamber in the hyperplane arrangement can each be chosen positive or negative.  That is, there is a generating set for $\CB$ such that each generator lives either in the subalgebra for the positive or negative part of the affine Grassmannian.  

Now, we observe that the homomorphism (\ref{eq: CB shift 1}) is an isomorphism on the subalgebras corresponding to the positive part of the affine Grassmannian, as in \cite[Remark 3.15]{BFN2}.  Similarly (\ref{eq: CB shift 2}) is an isomorphism for the negative part of the affine Grassmannian.  It follows from the existence of the generating set above that the images of (\ref{eq: CB shift 1}) and (\ref{eq: CB shift 2}) contain a full set of generators for $\CB$.  By the above diagrams, these images are all contained in the image of $\bY_\mu[z_1,\ldots,z_N] \rightarrow \CB$, so  this map is surjective.
\end{proof}

\subsection{Embedding into difference operators, and OGZ algebras}
\label{section: OGZ}

Let $\sD$ be the $\CC[z_1,\ldots,z_N]$--algebra with generators $\eV_{i,r}, \shV_{i,r}^{\pm 1}$ and $(\eV_{i,r} - \eV_{i,s} - n)^{-1}$, where $i \in I$, $1\leq r \neq s \leq \bv_i$ and $n \in \ZZ$.  The relations are that $\shV_{i,r} \eV_{j, s} = (\eV_{j,s} + \delta_{i,j} \delta_{r,s}) \shV_{i,r}$, while all other elements commute.

There is an injective homomorphism $\bz^\ast (\iota_\ast)^{-1} : \CBone \hookrightarrow \sD$, defined in \cite[Remark 5.23]{BFN1} and studied in \cite[Appendix A(i)--(ii)]{BFN2}, which is essentially the {\bf Abelianization map} from the physics literature \cite[Section 4.2]{BDG}.  The images of minuscule dressed monopole operators can be explicitly computed, and in particular the images of (\ref{eq: Y to CB 1}--\ref{eq: Y to CB 3}) are, using Yangian notation,
\begin{align}
A_i^{(r)} & \mapsto (-1)^r e_r(\eV_{i,1},\ldots, \eV_{i,\bv_i} ), \label{eq: difference ops 1}\\
E_i^{(1)} & \mapsto - \sum_{r=1}^{\bv_i} \prod_{k: i_k = i}(\eV_{i,r} - z_k - \tfrac{1}{2}) \frac{\prod_{j\rightarrow i} \prod_{s=1}^{\bv_j} (\eV_{i,r} - \eV_{j,s} - \tfrac{1}{2} ) }{\prod_{s=1, s\neq r}^{\bv_i}(\eV_{i,r} - \eV_{i,s}) } \shV_{i,r}^{-1} \label{eq: difference ops 2}\\
F_i^{(1)} & \mapsto \sum_{r=1}^{\bv_i} \frac{\prod_{i\rightarrow j} \prod_{s=1}^{\bv_j} (\eV_{i,r} - \eV_{j,s} + \tfrac{1}{2}) }{\prod_{s=1, s\neq r}^{\bv_i} (\eV_{i,r} - \eV_{i,s})} \shV_{i,r} \label{eq: difference ops 3}
\end{align}
This calculation appears in \cite[Appendix B]{BFN2}.  As a corollary of the previous two sections, 

\begin{Corollary}
$\bz^\ast(\iota_\ast)^{-1}$ identifies $\CBone$ with the $\CC[z_1,\ldots,z_N]$--subalgebra of $\sD$ generated by the elements (\ref{eq: difference ops 1}--\ref{eq: difference ops 1}), taken over all $i\in I$.
\end{Corollary}

This result may be considered as another definition of $\CBone$, which is analogous to the definition of {\bf orthogonal Gelfand-Zetlin algebras} due to Mazorchuk \cite{Maz}.  More precisely, let us take our quiver $Q$ to be the orientation 
$$1\leftarrow \cdots \leftarrow n-1$$ 
of the Dynkin diagram of type $A_{n-1}$, and assume our graded vector space $W = W_{n-1}$. In particular $N = \dim_\CC W_{n-1} = \bw_{n-1}$.  To more closely align with the notation of \cite[Section 5.1]{Web6}, denote
\begin{align}
x_{i, k} & = - \eV_{i,k} - \tfrac{i}{2},  & \text{ for } 1\leq i \leq n-1\text{ and } 1\leq k \leq \bv_i,  \\
x_{n,k} & = -z_k-\tfrac{n}{2}, & \text{ for } 1\leq k \leq N, \\
\varphi_{i,k}^{\pm 1} & = \shV_{i,k}^{\mp}, & \text{ for } 1\leq i \leq n-1\text{ and } 1\leq k \leq \bv_i
\end{align}
Note that $\varphi_{i,k} x_{j,\ell} = (x_{j,\ell} + \delta_{i,j} \delta_{k,\ell}) \varphi_{i,k}$.    Define $\br = (r_1,\ldots, r_n)$ by $r_i = \bv_i$ for $1\leq i \leq n-1$ and $r_n = N = \bw_{n-1}$.   Finally, denote the right-hand side of (\ref{eq: difference ops 2}) by $X_i^+$, and of (\ref{eq: difference ops 3}) by $X_i^-$.  Explicitly, we have
\begin{align}
X_i^+ & =  -\sum_{k=1}^{r_i} \frac{\prod_{\ell=1}^{r_{i+1}} (x_{i+1, \ell} - x_{i,k})}{\prod_{\ell =1, \ell \neq k}^{r_i} (x_{i,\ell} - x_{i,k})} \varphi_{i,k}, \\
X_i^- &= \sum_{k=1}^{r_i} \frac{ \prod_{\ell=1}^{r_{i-1}} (x_{i-1, \ell} - x_{i,k}) }{\prod_{\ell=1, \ell \neq k}^{r_i} (x_{i,\ell} - x_{i,k}) } \varphi_{i,k}^{-1}
\end{align}
Up to a sign, these are equal to the same-named elements from \cite[Section 5.1]{Web6}.   
%
%

We define the OGZ algebra $U(\br) \subset \sD$ as the subalgebra generated by all $X_i^\pm$ and all symmetric polynomials $\CC[x_{i,r}]^{\Gamma}$, where $\Gamma = \prod_{1\leq i \leq n} \Sigma_{r_i}$.  We may slightly enlarge this algebra by allowing arbitrary polynomials in $x_{n,k}$ for $1\leq k \leq r_n$ (note that these elements are central). Denoting this enlarged algebra by $\widetilde{U}(\br) \subset \sD$, we have $U(\br) = \widetilde{U}(\br)^{\Sigma_{r_n}}$ under the natural action.  
\begin{Corollary}
The map $\bz^\ast (\iota_\ast)^{-1}$ defines an isomorphism $\CBone \stackrel{\sim}{\rightarrow} \widetilde{U}(\br)$.
\end{Corollary}

The astute reader will note that our algebra $\widetilde{U}(\br)$ is {\em not quite} the OGZ algebra as defined in \cite[Section 3]{Maz}, \cite[Section 4.4]{Hartwig}: in both of these works, the factors $\varphi_{i,k}^\pm$ of $X_i^\pm$  appear on the {\em left} instead of the right.  In fact, this defines an isomorphic algebra: 

We may think of $\widetilde{U}(\br)$ as the image of $Y_\mu[z_1,\ldots,z_N] \rightarrow \sD$. There is an anti-involution of the shifted Yangian defined by $E_i^{(1)} \leftrightarrow F_i^{(1)}$, fixing the generators $A_i^{(r)}$ and $z_k$.  There is also an anti-involution of $\sD$ defined by $\varphi_{i,k}^\pm \mapsto \varphi_{i,k}^\mp$, fixing the generators $x_{i,k}$. Composing our map $Y_\mu[z_1,\ldots, z_N] \rightarrow \sD$ with both anti-involutions, we get another map $Y_\mu [z_1,\ldots, z_N ] \rightarrow \sD$, which has the property that in the images of $E_i^{(1)}, F_i^{(1)}$ the terms $\varphi_{i,k}^\pm$ appear on the left.  

The image of this new homomorphism $Y_\mu[z_1,\ldots,z_N] \rightarrow \sD$ is isomorphic to $\widetilde{U}(\br)$, but has a further discrepancy from \cite{Maz} and \cite{Hartwig}: $X_i^\pm$ now contain terms with $x_{i\mp 1, k}$ in their numerators, instead of $x_{i\pm 1, k}$. To rectify this, when defining $\CBone$ we take the opposite orientation of our quiver $1\rightarrow 2\rightarrow \cdots \rightarrow n-1$, and also choose the orientation $\Hom(V_{n-1}, W_{n-1})$ in the definition of $\bN$.  By \cite[Section 6(viii)]{BFN1}, the resulting algebra is isomorphic. This shows that our algebra ${U}(\br)$ is indeed isomorphic to the OGZ algebra from \cite{Maz}, \cite{Hartwig}.

\section{Coulomb branch categories}
\label{sec: coulomb cat}

In this section, following Webster \cite[Section 3]{Web3}, we  define a category such that the algebra $\ICB$ appears as an endomorphism algebra.  Note that there are many variations on this category, and we have chosen rather plain conventions, in order to simplify our goal of relating this category with the KLR Yangian algebra $\Ya$ from \cite[Section 4.4]{KTWWY2}.

\subsection{Webster's extended BFN category}

Denote by $\bt_\RR \subset \bt = \operatorname{Lie}\bT$ the set of diagonal matrices with entries in $\RR$.  For a point $\eta = (\eta_{i, r}) \in \bt_\RR$, we can associate linear operators on $\bg((z))$ and $\NK$, defined by the (adjoint) action of $\eta + z\tfrac{\partial}{\partial z}$.  The eigenvalues of these operator are in $\RR$.  We define $\Iwa_\eta \subset \GK$ to be the Iwahori subgroup whose Lie algebra is the sum of the non-negative eigenspaces. Similarly, we define $U_\eta\subset \NK$ to be sum of the non-negative degree eigenspaces.  It is easy to see that $\Iwa_\eta$ preserves the subspace $U_\eta$.

For generic $\eta$, note that $\Iwa_\eta$ only depends on which alcove $\eta$ lies in, while $U_\eta$ only depends on $\eta$'s cell in a hyperplane arrangement determined by the weights of $\bN$.  Following \cite[Definition 5.2]{BFN1} let us define a {\bf generalized root} for the pair $(\bG, \bN)$ to be either (i) a root of $\operatorname{Lie}\bG$ or (ii) a weight of $\bN$.  	Denote by $\objs \subset \bt_\RR$ the complement of all {integer shifts} of generalized root hyperplanes.  Explicitly, these are the hyperplanes
\begin{align}
\eV_{i,r} - \eV_{i,s} & = n, \ \ \text{ for } i\in I, \text{ all }r\neq s, \text{ and } n\in \ZZ  \label{eq: hyper 1}\\
\eV_{i,r} & = n, \ \ \text{ for } i\in I,\text{ all }  r, \text{ and } n \in \ZZ, \label{eq: hyper 2}\\
\eV_{i,r} - \eV_{j,s} & = n, \ \ \text{ for } i \sim j, \text{ all }r,s, \text{ and } n \in \ZZ \label{eq: hyper 3}
\end{align}
Their complement $\objs$ is a disjoint union of cells. By the above, the pair $(\Iwa_\eta, U_\eta)$ depends only on the cell of $\eta$ in $\objs$.

\begin{Remark}
As in \cite{Web3}, we could incorporate non-integral shifts above.  For example, this is desireable in situations where some weight of $\bN$ coincides with a root of $\bG$.  This does not occur for a quiver without loops, so we omit this complication.
\end{Remark}

For any $\eta, \eta' \in \objs$, we define analogues of (\ref{eq: R ICB def}) and (\ref{eq: ICB def}),
\begin{align}
\catR{\eta'}{\eta} & \DEF  \left\{ [g, s] \in \GK \times^{\Iwa_\eta} U_\eta :  gs \in U_{\eta'} \right\}, \\
\CCB{\eta'}{\eta} &\DEF H^{(\Iwa_{\eta'} \times \bF)\rtimes \CC^\times}_\ast( \catR{\eta'}{\eta} )
\end{align}
The group action and Borel-Moore homology are defined as described in Section \ref{section: quantized Coulom branch algebras}.  An analogue of Theorem \ref{thm: bfn thm} holds, and in particular there is a convolution product 
\begin{equation}
\label{eq: conv cat}
\CCB{\eta''}{\eta'} \otimes \CCB{\eta'}{\eta} \longrightarrow \CCB{\eta''}{\eta}
\end{equation}

Following \cite[Definition 3.5]{Web3}, we take:
\begin{Definition}
The {\bf extended BFN category} $\sB$ has objects the elements $\eta \in \objs$, and morphisms 
$$
\Hom_{\sB}( \eta, \eta') \DEF \CCB{\eta'}{\eta},
$$
with composition of morphisms given by the convolution product (\ref{eq: conv cat}).
\end{Definition}

There is an action of the affine Weyl group $\widehat{\Sigma} = \Sigma \ltimes \ZZ^\bv$ on $\bt$, which preserves $\objs$: the finite Weyl group $\Sigma$ acts via its standard permutation action on $\bt$, while the lattice $\ZZ^\bv$ acts by $\laG \cdot \eta = \eta - \laG$.  With these conventions,  it is not hard to see that
\begin{equation}
\label{eq: Weyl invariance}
\Iwa_{w \cdot \eta} = w \Iwa_\eta w^{-1}, \qquad U_{w\cdot \eta} = w U_\eta,
\end{equation}
where $w\in \widehat{\Sigma}$ acts on $\GK$ and $\NK$ via the embedding $\widehat{\Sigma} \subset \GK$ from Section \ref{sec: Gr and Fl}. 

\subsection{A presentation of \texorpdfstring{$\sB$}{B}}
\label{sec: presentation of cat}

In \cite[Section 3.2]{Web3}, Webster describes the morphisms in $\sB$ in terms of paths in $\bt$.   Let $\paths \subset \bt_\RR$ denote the complement of all pairwise intersections of hyperplanes (\ref{eq: hyper 1}--\ref{eq: hyper 3}).  Then a {\bf path} will mean a piecewise smooth map $\pi: [0,1] \rightarrow \paths$, which is transverse to any hyperplane that it crosses.  We compose paths in the usual way, if their end and start agree.  Precise parametrizations will generally not be important.

Just as $(\Iwa_\eta, U_\eta)$ only depend on the cell of $\eta$ in $\objs$, the morphism corresponding to a path $\pi$ will only depend on the sequence $H_1,\ldots, H_\ell$ of hyperplanes  that it crosses, in order.  We call the total number $\ell$ of hyperplanes that $\pi$ crosses its {\bf length}.  Suppose that $\pi$ crosses the hyperplane $H_k$ at ``time'' $t_k\in (0,1)$, and choose some intermediate times $t_1 < s_1 < t_2 < s_2 <\ldots < s_{\ell - 1} < t_\ell$.  We write $s_0 = \pi(0)$ and $s_\ell = \pi(1)$.  Denote $\eta_k^- = \pi(s_{k-1})$ and $\eta_k^+ = \pi(s_k)$, with associated $\Iwa_k^\pm, U_k^\pm$.   
In other words, $\eta_k^\pm$ lie in adjacent cells across the hyperplane $H_k$, with corresponding pairs $(\Iwa_k^\pm, U_k^\pm)$.
With these conventions, the morphism associated to the hyperplane $H_k$ is of one of two types:

\begin{enumerate}
\item[(a)] If the hyperplane $H_k$ corresponds to a weight of $\bN$ (cases (\ref{eq: hyper 2}--\ref{eq: hyper 3})) then $\Iwa_k^+ = \Iwa_k^-$, while $U_k^\pm$ differ by one dimension.  We let 
\begin{equation*}
r(\eta_k^+, \eta_k^-) = [ \pi^{-1}(1) ]  = \Hom_{\sB}(\eta_k^-, \eta_k^+)
\end{equation*}
be the fiber over the unit point of $\GK / \Iwa_k^-$, under the map $\pi : \catR{\eta_k^+}{\eta_k^-} \rightarrow \GK / \Iwa_k^-$ sending $[g, s] \mapsto [g]$ as usual.  
\item[(b)] If the hyperplane $H$ correspond to a root of $\bG$ (case (\ref{eq: hyper 1})) then $U_k^+ = U_k^-$, while the Iwahoris $\Iwa_k^\pm$ differ by a single (affine) root $\alpha$. In this case we let
\begin{equation*}
u_\alpha = [ \pi^{-1} \big( \overline{ \Iwa_k^+ \Iwa_k^- / \Iwa_k^-} \big) ]  \in \Hom_{\sB}(\eta_k^-, \eta_k^+)
\end{equation*}
Note that $\overline{ \Iwa_k^+ \Iwa_k^- / \Iwa_k^-} \cong \mathbb{P}^1$ by our assumption.
\end{enumerate}
In this way, we have a morphism $x_k$ for each $H_k$.   Finally, we define the morphism
\begin{equation}
\label{eq: def of r from hyperplane sequence}
\mathbbm{r}_\pi = x_k x_{k-1} \cdots x_1  \ \in \Hom_{\sB}\big(\pi(0), \pi(1) \big)
\end{equation}
This agrees with \cite[Definition 3.11]{Web3}, because of \cite[Relation (3.4e)]{Web3}.  

Besides the elements $\mathbbm{r}_\pi$, for any $w\in \widehat{\Sigma}$ there is an isomorphism 
\begin{equation*}
y_w = [ \pi^{-1}(w) ] \in \Hom_{\sB}(\eta, w \eta)
\end{equation*}
where we think of $w \in \GK / \Iwa_\eta$ as in Section \ref{sec: Gr and Fl}.

Webster gives a complete set of relations between the above elements \cite[Theorem 3.10]{Web3}, and also gives bases for any $\Hom_{\sB}(\eta, \eta')$ \cite[Corollary 3.12]{Web3}.  For our present purposes it will suffice to know that
\begin{equation}
\label{eq: paths and weyl group}
y_w \mathbbm{r}_\pi = \mathbbm{r}_{w \cdot \pi} y_w
\end{equation}
where $\widehat{\Sigma}$ acts on the set of paths pointwise, via its action on $\bt$.

\begin{Remark}
\cite[Theorem 3.10]{Web3} also provides a faithful action of $\sB$ where each $\eta \in \objs$ is assigned a copy of the polynomial ring $H_{\bT\times \bF\times \CC^\times}^\ast(pt)$, which is helpful to keep in mind.  Very roughly speaking, acting in this polynomial representation (i) $r(\eta, \eta')$ is multiplication by a linear polynomial, (ii) $u_\alpha$ is a divided difference operator in the affine root $\alpha$, and (iii) $y_w$ acts by $w$.
\end{Remark}

\subsection{A subcategory}
\label{section: a subcategory}
Recall that we denote $|\bv| = \sum_i \bv_i$. Consider the set 
\begin{equation}
I^\bv \DEF \Big\{ \bi = (i_1,\ldots, i_m) \in I^{|\bv|} : \#\{k : i_k = i\} = \bv_i \text{ for all } i\Big\}
\end{equation}
An element $\bi \in I^\bv$ is equivalent to a total order $<_{\bi}$ on the set of pairs
\begin{equation*}
\left\{ (i, r) : i \in I, 1\leq r \leq \bv_i \right\}
\end{equation*}
which restricts to the order $(i,1) <_{\bi} (i,2)<_\bi \ldots < _\bi(i, \bv_i)$ for each $i\in I$.  

\begin{Example}
If $I = \{1, 2\}$ and $\bv_1 = 3, \bv_2 = 2$, then $\bi = (1,2,1,2,1) \in I^{\bv}$.  It corresponds to the total order
$$
(1,1) <_\bi (2,1) <_\bi (1,2) <_\bi (2,2) <_\bi (1, 3)
$$
\end{Example}

For each $\bi$, choose $\eta = (\eta_{i,r})_{\substack{i\in I, 1\leq r\leq \bv_i}} \in \objs$ with components\footnote{For simplicity, we will choose the $\eta_{i,r}$ to partition the interval $[0,1]$ into equal parts.} $0<\eta_{i,r} <1$, and such that $(i,r) <_{\bi} (j, s)$ implies $\eta_{i,r} < \eta_{j,s}$.   Then
\begin{equation}
U_{\eta_\bi} =  \label{eq: def of subspace for bi}
\end{equation}
\begin{equation*}
\bigoplus_{i\rightarrow j} \Big( \bigoplus_{(i,r) <_\bi (j, s)} \Hom_\CC( V_{i,r}, V_{j,s} )[[z]] \oplus \bigoplus_{(i,r) >_\bi (j,s)} z \Hom_\CC( V_{i,r}, V_{j,s} )[[z]] \Big) \oplus \bigoplus_i \Hom_\CC(W_i, V_i)[[z]] 
\end{equation*}
where we denote the lines $V_{i,r} =\CC e_{i,r} \subset V_i$.  Meanwhile, for {any} $\bi\in I^\bv$ the Iwahori $\Iwa_{\eta_\bi} = \Iwa$ is always our usual one (\ref{eq: the Iwahori}).

\begin{Definition}
The category $\catya$ is the full subcategory of $\cB$ having objects $I^\bv$ (under $\bi \mapsto \eta_\bi$).
\end{Definition}

\begin{Remark}
\label{remark: finkelberg icm}
Assume that $W=0$.  Then the category $\catya$ appears in \cite[Section 6.5]{F}.  Following the notation there, for any $\bi \in I^\bv$ we can associate a full flag $V = V^0 \supset V^1 \supset \ldots \supset V^N = 0$ where $N = |\bv| = \dim V$.  It is defined with respect to our basis $\basV_{i,r}$, using the order $<_\bi$:, $V^1$ is the span of all $\basV_{i,r}$ except the least $(i,r)$, $V^2$ is the span of all except the least two $(i,r)$, etc.    With the above conventions, one can verify that the algebra $\mathcal{H}_V = \bigoplus_{\bi,\bj \in I^\bv} {_\bj \mathcal{H}_\bi}$ from \cite{F} is equal to $\bigoplus_{\bi,\bj\in I^\bv} \Hom_{\catya}(\bi,\bj)$.  
\end{Remark}

The categories $\sB$ and  $\catya$ are closely related to the algebra $\ICB$ from (\ref{eq: ICB def}).  To see this, first fix any $\bi_{\bv} \in I^{\bv}$ such that $i\rightarrow j$ implies $(i, r) <_{\bi_\bv} (j,s)$ for all $r,s$. For example, by choosing a total order $i_1 < i_2 <\ldots$ on $I$ such that $i\rightarrow j$ imples $i<j$, we could define $\bi_\bv$ by
\begin{equation}
\label{eq: order and idempotent}
\bi_\bv = (\underbrace{i_1,\ldots, i_1}_{\bv_{i_1}\text{ times}}, \underbrace{i_2,\ldots,i_2}_{\bv_{i_2}\text{ times}},\ldots)
\end{equation}
\begin{Remark}
It is possible that no such total order exists, since $Q$ might have an oriented cycle. But by \cite[Section 6(viii)]{BFN1}, if we reorient an arrow of $Q$ then $\CB$ is unchanged up to isomorphism. Thus we can safely assume that $Q$ is acyclic, so that a total order does exist.
\end{Remark}

A similar element (denoted $\bi_\bm$) is considered in \cite[Section 4.5]{KTWWY2}, but with additional requirements in its definition since there the Dynkin diagram $I$ is assumed bipartite.  (We do not make this bipartiteness assumption here, as it is not relevant to our study.)

From the explicit description (\ref{eq: def of subspace for bi}) we see that  $\subrep{\bi_{\bv}} = \NO$.  Since $\Iwa_{\eta_\bi} =\Iwa$ for all $\bi\in I^\bv$, there is an equality $ \cR = \catR{\bi_\bv}{\bi_\bv}$ and it is $(\Iwa\times\bF)\rtimes \CC^\times$--equivariant. Since the convolution products are also defined in the same way, it follows that
\begin{Lemma}
\label{prop: ICB in CCB}
For any choice of $\bi_\bv$ as above, there is an equality of algebras $\ICB = \End_{\catya}( \bi_\bv ) $.
\end{Lemma}

\subsection{Cylindrical diagrams and morphisms} 
\label{sec: Cylindrical diagrams and morphisms} 
For the category $\catya$, we can encode the paths from Section \ref{sec: presentation of cat} using the {\bf cylindrical KLR diagrams} from \cite[Section 4.4]{KTWWY2}.  Recall that a cylindrical KLR diagram $D$ consists of a finite number of curves in a cylinder $\RR / \ZZ \times [0, 1]$, locally of the form
$$
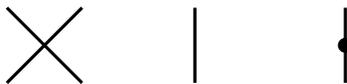
\begin{tikzpicture}
  \draw[very thick] (-2,0) +(-.5,-.5) -- +(.5,.5);
  \draw[very thick](-2,0) +(.5,-.5) -- +(-.5,.5);


 \draw[very thick](0,0) +(0,-.5) --  +(0,.5);

  \draw[very thick](2,0) +(0,-.5) --  node
  [midway,circle,fill=black,inner sep=2pt]{}
  +(0,.5);
\end{tikzpicture}
$$
which meet the ``bottom circle'' $t=0$ and ``top circle'' $t=1$ at distinct points with $x\neq 0$, where $(x,t)$ are coordinates on $ \RR / \ZZ \times [0,1]$.  Each curve carries a label $i \in I$.  See the left-hand side of (\ref{diagram: diagram}) for an example (which also depicts the coordinate axis $(x,t)$), and \cite{KTWWY2} for more details.

At the bottom of $D$, the sequence of labels of the curves defines a sequence $\operatorname{bot}(D) \in I^m$ in order of increasing $x$. Similarly we get $\operatorname{top}(D) \in I^m$.  Also, at the top or bottom of $D$, the list of $x$--coordinates of the curves encode a point of $\objs$, so that the coordinate of the $r$th curve labelled $i$ is $\eta_{i,r}$. {\em We will restrict our attention to those diagrams where $\operatorname{bot}(D), \operatorname{top}(D) \in I^\bv$, and such that the $x$--coordinates of the curves at the top and bottom are precisely given by $\eta_{\operatorname{bot}(D)}$ and $\eta_{\operatorname{top}(D)}$, respectively.}

\begin{equation}
\label{diagram: diagram}
\begin{tikzpicture}[baseline]

\draw[thick,->] (-4,0) +(-1,0) -- +(-1,0.5) node[at end, below left] {$t$};
\draw[thick,->] (-4,0) +(-1,0) -- +(-0.5,0) node[at end, below left] {$x$};
\draw[very thick, dashed] (-4,0) +(0,-1) -- +(0,1);
\draw[very thick, dashed] (-4,0) +(4,-1) -- + (4,1);
\draw[very thick] (-4,0) +(1,-1) -- +(2,1) ;
\draw[very thick] (-4,0) +(2,-1) to[out=110,in=-20] +(0,.1) ;
\draw[very thick] (-4,0) + (4,.1) to[out=160,in=-70] + (3,1) ;
\draw[very thick] (-4,0) +(3,-1) to[out=50,in=-160] +(4,-0.5);
\draw[very thick] (-4,0) +(0,-0.5) to[out=20,in=-100] +(1,1);

\draw (-4,0) + (1,-1.3) node {$i$};
\draw (-4,0) +(2,-1.3) node {$i$};
\draw (-4,0) +(3,-1.3) node {$j$};
\draw (-4,0) +(1,1.3) node {$j$};
\draw (-4,0) +(2,1.3) node {$i$};
\draw (-4,0) +(3,1.3) node {$i$};

\end{tikzpicture} \quad
\stackrel{\text{unroll}}{\relbar\joinrel\relbar\joinrel\longrightarrow} \quad 
\begin{tikzpicture}[baseline]

\draw[very thick] (-4,0) +(1,-1) -- +(2,1) ;
\draw[very thick] (-4,0) +(2,-1) to[out=110,in=-20] +(0,.1) ;
\draw[very thick] (-4,0) + (0,.1) to[out=160,in=-70] + (-1,1) ;
\draw[very thick] (-4,0) +(3,-1) to[out=50,in=-160] +(4,-0.5);
\draw[very thick] (-4,0) +(4,-0.5) to[out=20,in=-100] +(5,1);

\draw (-4,0) + (1,-1.3) node {$(i,1)$};
\draw (-4,0) +(2,-1.3) node {$(i,2)$};
\draw (-4,0) +(3,-1.3) node {$(j,1)$};
\draw (-4,0) +(5 	,1.3) node {$(j,1)$};
\draw (-4,0) +(-1,1.3) node {$(i,2)$};
\draw (-4,0) +(2,1.3) node {$(i,1)$};

\end{tikzpicture}
\end{equation}

For the time being, we focus on diagrams which carry no dots.  Looking at the bottom of $D$, consider the $r$th curve having label $i$. It may be lifted uniquely to a path $\pi_{i,r} : [0, 1] \rightarrow \RR$ in the universal cover of $\RR \rightarrow \RR/\ZZ$, with the property $\pi_{i,r}(0) = \eta_{i,r}$.  These lifts define an {\bf unrolled path} $\pi = (\pi_{i,r}) : [0,1] \rightarrow \paths$, associated to $D$.  

\begin{Lemma}
For $D, \pi $ as above, there exists a unique $w\in \widehat{W}$ such that $\pi(1) = w \cdot \eta_{\operatorname{top}(D)}$.
\end{Lemma}

So, from our diagram $D$ we get a pair $(\pi, w)$. Denote the corresponding morphisms
\begin{align*}
\mathbbm{r}_{D} & \DEF \mathbbm{r}_\pi \in \Hom_{\sB}(\eta_{\operatorname{bot}(D)}, w \cdot \eta_{\operatorname{top}(D)}), \\
y_D & \DEF y_w \in \Hom_{\sB} (\eta_{\operatorname{top}(D)}, w\cdot \eta_{\operatorname{top}(D)})
\end{align*}
Their composition gives a morphism in $\catya$: 
\begin{equation}
\morphism_D \DEF  y_D^{-1} \mathbbm{r}_D \in \Hom_{\catya}(\operatorname{bot}(D), \operatorname{top}(D) )
\end{equation}

\begin{Remark}
The reader may have noticed that we have secretly slightly generalized our notion of path, since cylindrical KLR diagrams might have several crossings at the same ``time'' $t$:
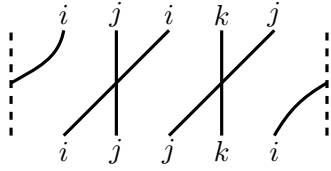
$$
\begin{tikzpicture}[scale=0.7]
\draw[very thick, dashed] (0,0) -- (0,2);
\draw[very thick, dashed] (6,0) -- (6,2);

\draw[very thick] (1,0) -- (3,2);
\draw[very thick] (2,0) -- (2,2);
\draw[very thick] (3,0) -- (5,2); 
\draw[very thick] (4,0) -- (4,2);
\draw[very thick] (5,0) to[out=60,in=-150] (6,1);
\draw[very thick] (0,1) to[out=30,in=-100] (1,2);

\draw (1,-0.3) node {$i$};
\draw (2,-0.3) node {$j$};
\draw (3,-0.3) node {$j$};
\draw (4,-0.3) node {$k$};
\draw (5,-0.3) node {$i$};

\draw (1,2.3) node {$i$};
\draw (2,2.3) node {$j$};
\draw (3,2.3) node {$i$};
\draw (4,2.3) node {$k$};
\draw (5,2.3) node {$j$};

\end{tikzpicture}
$$
Here, the coresponding path $\pi$ passes through an intersection of three hyperplanes.  Thankfully, the discussion below shows that we can isotope $D$, without changing the corresponding morphisms in $\mathscr{B}$.  So the reader may either accept these more general paths (which may pass through intersections of any number of hyperplanes, so long as the hyperplanes are in disjoint sets of variables), or else restrict to those $D$ where crossings never occur at the same height.
\end{Remark}

Cylindrical KLR diagrams $D', D''$ can be multiplied if $\operatorname{bot}(D') = \operatorname{top}(D'')$: we define $D' \cdot D''$ by stacking $D'$ on top of $D''$.   We may also consider isotopies of cylindrical KLR diagrams, with the extra requirements that (a) all isotopies fix the top and bottom circles, and (b) they do not create or break crossings, {\em including of the ``seam'' $x=0$}.

\begin{Proposition}
\label{lemma: where to put the Weyl group} For any cylindrical KLR diagram $D$ without dots,
\begin{enumerate}
\item[(a)]  $\mathbbm{r}_D$ and $y_D$ only depend on $D$ up to isotopy, defined as above.  In particular, $\morphism_D$ only depends on $D$ up to isotopy.

\item[(b)] If $D = D'\cdot D''$, then $\morphism_D = \morphism_{D'} \circ \morphism_{D''}$.

\end{enumerate}
\end{Proposition}
\begin{proof}
Both parts are consequences of \cite[Theorem 3.10]{Web3}.  In particular part (b) follows from (\ref{eq: paths and weyl group}).  

Consider then part (a).  The claim about $y_D$ is well-kwown, and follows from the relations of $\widehat{\Sigma}$.  For $\mathbbm{r}_D$, first suppose that $D$ has no crossings at the same time $t$, and consider an isotopy $D\rightarrow D'$ which preserves this property.  Then the corresponding paths pass through precisely the same sequence of cells and hyperplanes in $\paths$, so by the definition (\ref{eq: def of r from hyperplane sequence} we have that $\mathbbm{r}_D = \mathbbm{r}_{D'}$.  It remains to check that subdiagrams can ``slide'' past each other in time, if they are formed by disjoint sets of strands:
\begin{equation}
\label{eq: sldes are fine}
\begin{tikzpicture}[baseline,scale=0.8]
\draw[very thick] (0,-1.1)+(-0.3,0.5) -- +(1.3,0.5) -- +(1.3,1) -- +(-0.3,1) -- +(-0.3,0.5);
\draw[very thick] (0,-1.1) +(0,0) -- +(0,0.5);
\draw[very thick] (0,-1.1) +(1,0) -- +(1,0.5);
\draw[very thick] (0,-1.1) +(1,1) -- +(1,2.3);
\draw[very thick] (0,-1.1) +(0,1) -- +(0,2.3);
\draw (0,-1.1) +(0.5,0.2) node {$\cdots$};
\draw (0,-1.1) +(0.5,1.6) node {$\cdots$};

\draw[very thick] (2,-0.3)+(-0.3,0.5) -- +(1.3,0.5) -- +(1.3,1) -- +(-0.3,1) -- +(-0.3,0.5);
\draw[very thick] (2,-0.3)+(0,1) -- +(0,1.5);
\draw[very thick] (2,-0.3)+(1,1) -- +(1,1.5);
\draw[very thick] (2,-0.3)+(0,-0.8) -- +(0,0.5);
\draw[very thick] (2,-0.3)+(1,-0.8) -- +(1,0.5);
\draw[very thick] (2,-0.3) + (0.5,1.2) node {$\cdots$};
\draw[very thick] (2,-0.3) + (0.5,-0.2) node {$\cdots$};
\end{tikzpicture}
\quad = \quad
\begin{tikzpicture}[baseline,scale=0.8]
\draw[very thick] (2,-1.1)+(-0.3,0.5) -- +(1.3,0.5) -- +(1.3,1) -- +(-0.3,1) -- +(-0.3,0.5);
\draw[very thick] (2,-1.1) +(0,0) -- +(0,0.5);
\draw[very thick] (2,-1.1) +(1,0) -- +(1,0.5);
\draw[very thick] (2,-1.1) +(1,1) -- +(1,2.3);
\draw[very thick] (2,-1.1) +(0,1) -- +(0,2.3);
\draw (2,-1.1) +(0.5,0.2) node {$\cdots$};
\draw (2,-1.1) +(0.5,1.6) node {$\cdots$};

\draw[very thick] (0,-0.3)+(-0.3,0.5) -- +(1.3,0.5) -- +(1.3,1) -- +(-0.3,1) -- +(-0.3,0.5);
\draw[very thick] (0,-0.3)+(0,1) -- +(0,1.5);
\draw[very thick] (0,-0.3)+(1,1) -- +(1,1.5);
\draw[very thick] (0,-0.3)+(0,-0.8) -- +(0,0.5);
\draw[very thick] (0,-0.3)+(1,-0.8) -- +(1,0.5);
\draw[very thick] (0,-0.3) + (0.5,1.2) node {$\cdots$};
\draw[very thick] (0,-0.3) + (0.5,-0.2) node {$\cdots$};
\end{tikzpicture}
\quad \stackrel{def}{=} \quad
\begin{tikzpicture}[baseline,scale=0.8]
\draw[very thick] (2,-1.1)+(-0.3,1) -- +(1.3,1) -- +(1.3,1.5) -- +(-0.3,1.5) -- +(-0.3,1);
\draw[very thick] (2,-1.1) +(0,0) -- +(0,1);
\draw[very thick] (2,-1.1) +(1,0) -- +(1,1);
\draw[very thick] (2,-1.1) +(1,1.5) -- +(1,2.3);
\draw[very thick] (2,-1.1) +(0,1.5) -- +(0,2.3);
\draw (2,-1.1) +(0.5,0.5) node {$\cdots$};
\draw (2,-1.1) +(0.5,1.9) node {$\cdots$};

\draw[very thick] (0,-1.1)+(-0.3,1) -- +(1.3,1) -- +(1.3,1.5) -- +(-0.3,1.5) -- +(-0.3,1);
\draw[very thick] (0,-1.1) +(0,0) -- +(0,1);
\draw[very thick] (0,-1.1) +(1,0) -- +(1,1);
\draw[very thick] (0,-1.1) +(1,1.5) -- +(1,2.3);
\draw[very thick] (0,-1.1) +(0,1.5) -- +(0,2.3);
\draw (0,-1.1) +(0.5,0.5) node {$\cdots$};
\draw (0,-1.1) +(0.5,1.9) node {$\cdots$};
\end{tikzpicture}
\end{equation}
This is a consequence of \cite[Theorem 3.10]{Web3}, for a very simple reason: in the faithful polynomial representation on $\sB$ given there, crossings correspond to divided difference operators, difference operators, or multiplication by polynomials.  Moreover, the set of variables on which a crossing acts is encoded by its strands.  So in particular, the operations corresponding to our two subdiagrams take place in distinct sets of variables, and therefore commute.
\end{proof}

The simplest diagrams are those with a single crossing, either between two strands or between a strand and the ``seam'' of $\RR/\ZZ$.  Following the notation of \cite[Section 4.4]{KTWWY2}, we denote these by
\begin{equation*}
\psi_k(\bi)=\,\begin{tikzpicture}[baseline,scale=1.2]
      \draw[very thick, dashed] (4,0) +(0,-.5) -- +(0,.5);
 \draw[very thick] (4.3,0) +(0,-.5) -- +(0,.5);
  \draw[very thick] (6,0) +(0,-.5) -- +(0,.5);
  \draw[very thick, dashed] (6.3,0) +(0,-.5) -- +(0,.5);
  \draw (4.3,0) +(0.1,-.8) node {\small$i_1$};
  \draw (6,0) +(0.1,-.8) node {\small$i_{|\bv|}$};
  \draw[very thick] (4.9,0) +(0,-.5) -- +(.5,.5);
  \draw[very thick] (5.4,0) +(0,-.5) -- +(-.5,.5);
\draw (4.9,0) +(0.1,-.8) node {\small$i_k$};
\draw (5.4,0) +(0.1,-.8) node {\small$i_{k+1}$};
\draw (4.75,0) node {$\cdots$};
\draw (5.72,0) node {$\cdots$};
\end{tikzpicture} \qquad \qquad  
\sigma_+(\bi)=\,\begin{tikzpicture}[baseline,scale=1.2]
      \draw[very thick, dashed] (4,0) +(0,-.5) -- +(0,.5);
 \draw[very thick] (4.3,0) +(0,.5) -- +(-.3,0);
  \draw[very thick] (6,0) +(.3,0) -- node[at end, below] {\small$i_{|\bv|}$}+(0,-.5);
  \draw[very thick] (4.3,-.5) --node[at start, below] {\small$i_1$} (4.9,.5);
 \draw[very thick] (5.4,-.5) -- node[at start, below] {\small$i_{|\bv|-1}$} (6,.5);
  \draw[very thick, dashed] (6.3,0) +(0,-.5) -- +(0,.5);
\draw (5.15,0) node {$\cdots$};
    \end{tikzpicture} \qquad \qquad 
\sigma_-(\bi)=\,\begin{tikzpicture}[baseline,scale=1.2]
      \draw[very thick, dashed] (4,0) +(0,-.5) -- +(0,.5);
 \draw[very thick] (4.3,0) +(0,-.5) -- node[at start, below] {\small$i_1$} +(-.3,0);
  \draw[very thick] (6,0) +(.3,0) -- +(0,.5);
 \draw[very thick] (4.3,.5) --node[at end, below] {\small$i_2$} (4.9,-.5);
 \draw[very thick] (5.4,.5) --node[at end, below] {\small$i_{|\bv|}$} (6,-.5);
  \draw[very thick, dashed] (6.3,0) +(0,-.5) -- +(0,.5);
\draw (5.15,0) node {$\cdots$};
\end{tikzpicture}
\end{equation*}
The proposition shows that any $\morphism_D$ (with no dots) can be decomposed as a product of $\morphism$'s corresponding to these simpler diagrams.

We have thus far neglected the dots that a diagram $D$ might carry.  Consider the diagrams with straight vertical lines: 
$$
e(\bi)=\,\begin{tikzpicture}[baseline,scale=1.2]
\draw[very thick, dashed] (-3.7,0) +(0,-.5) -- +(0,.5);
 \draw[very thick] (-3.4,0) +(0,-.5) -- +(0,.5);
  \draw[very thick] (-2,0) +(0,-.5) -- +(0,.5);
  \draw[very thick, dashed] (-1.7,0) +(0,-.5) -- +(0,.5);
  \draw (-3.4,0) +(0.1,-.8) node {\small$i_1$};
  \draw (-2,0) +(0.1,-.8) node {\small$i_{|\bv|}$};
  \draw (-2.65,0) node {$\cdots$};  \end{tikzpicture}
\qquad \qquad   
\yz_k(\bi)=\,\begin{tikzpicture}[baseline,scale=1.2]

\draw[very thick, dashed] (0,0) +(0,-.5) -- +(0,.5);
 \draw[very thick] (.3,0) +(0,-.5) -- +(0,.5);
  \draw[very thick] (2,0) +(0,-.5) -- +(0,.5);
  \draw[very thick, dashed] (2.3,0) +(0,-.5) -- +(0,.5);
  \draw (.3,0) +(0.1,-.8) node {\small$i_1$};
  \draw (2,0) +(0.1,-.8) node {\small$i_{|\bv|}$};
  \draw (.75,0) node {$\cdots$};
    \draw[very thick] (1.2,0) +(0,-.5) -- node [midway,circle,fill=black,inner sep=2pt]{} +(0,.5);
     \draw (1.65,0) node {$\cdots$};
      \draw (1.2,0) +(0.1,-.8) node {\small$i_k$};  
\end{tikzpicture}
$$
To these diagrams, we associate the morphisms $1$ and $\eV_{i,r}$ in $H_{\bT}^\ast(pt) \subset \End_{\catya}(\bi)$, respectively, where as usual $(i,r)$ is determined by $i_k$ from $\bi$.  Appealing to \cite[Theorem 3.10]{Web3} again, we claim that these morphisms again isotope as expected, both with each other and with the $\morphism_D$ defined above.  We can therefore define 
\begin{equation}
\morphism_D \in \Hom_{\catya}(\operatorname{bot}(D), \operatorname{top}(D))
\end{equation}
for an arbitrary cylindrical KLR diagram $D$, and this is invariant under isotopy of $D$.

\subsection{Comparison with the algebra \texorpdfstring{\protect\Ya}{Ya}} 
\label{section: finale}

The {\bf KLR Yangian algebra} $\Ya$ is defined in \cite[Section 4.4]{KTWWY2}, as a sort of KLR algebra on a cylinder.  $\Ya$ is spanned by cylindrical KLR diagrams up to isotopy, modulo certain relations, and multiplication is given by stacking of diagrams.  For any $\bi \in I^n$ there is an idempotent $\idem(\bi) \in \Ya$, which is the diagram with straight vertical lines as above.  Note that in $\Ya$ the number of strands can be arbitrary, but we will restrict our attention to those diagrams with strands labelled by $I^\bv$.

To be more precise, in this paper we will consider $\Ya$ as an algebra over $H_{\bF \times \CC^\times}^\ast(pt) = \CC[\hbar,z_1,\ldots,z_N]$.  Thus we make the following changes in the definition:
\begin{enumerate}
\item[(a)] Homogenize relations \cite[(4.8a) and (4.8c)]{KTWWY2} by replacing the factors of $2$ by $\hbar$.

\item[(b)] Define $\overline{X}_{ij}(u,v) = u-v-\tfrac{1}{2}\hbar$ when $i\leftarrow j$.

\item[(c)] Let the roots of $p_i(u)$ be the generators of the polynomial ring $H_{\bF_i}^\ast(pt)$, so that $ p_i(u) = \prod_{k: i_k = i} (u-z_{k}) $.
   Also, define $p_{i,+}(u) = p_i(u+\hbar)$.
\end{enumerate}
We recover the algebra from \cite{KTWWY2} by taking $\hbar =2$, and specializing $z_1,\ldots,z_N$ to complex numbers.  The latter can be encoded by an $I$--tuple of multisets, denoted $\mathbf{R}$ in \cite{KTWWY2}.  

\begin{Theorem}
The map $D \mapsto \morphism_D$ defines an isomorphism of $H_{\bF\times \CC^\times}^\ast(pt)$--algebras
$$
\bigoplus_{\bi,\bj\in I^\bv} \idem(\bj) \Ya \idem(\bi)  \stackrel{\sim}{\longrightarrow} \bigoplus_{\bi, \bj \in I^\bv} \Hom_{\catya}(\bi, \bj)
$$
\end{Theorem}
\begin{proof}
We can verify that this map is well-defined by comparing the polynomial representations of $\Ya$ and $\sB$, defined in \cite[Theorem 4.17]{KTWWY2} and \cite[Theorem 3.10]{Web3}, respectively.  Note that in the conventions of \cite{KTWWY2}, for each $\bi$ we should identify 
$$
H_\bT^\ast(pt) = \CC	[\eV_{i,r} : i\in I, 1\leq r\leq \bv_i]  \ \cong \ \CC[Z_1(\bi), \ldots, Z_{|\bv|}(\bi) ]
$$
where $Z_1(\bi),\ldots,Z_{|\bv|}(\bi)$ map to the elements $\eV_{i,r}$ in increasing order under $<_\bi$.  Accounting for this convention change, we easily check that the action of the elements $\phi_k(\bi), \sigma_\pm(\bi), \idem(\bi)$ and $z_k(\bi)$ is the same whether thought of in $\Ya$ or in $\catya$.  This shows that the homomorphism is well-defined.

To show that this map is an isomorphism, we look at bases. For each $w\in \widehat{\Sigma}$, we fix a minimal length path $\pi$ from $\eta_\bi$ to $w \eta_\bj$, and consider the morphism $y_w^{-1} \mathbbm{r}_\pi \in \Hom_{\catya}(\bi, \bj)$.   On the one hand, by \cite[Corollary 3.12]{Web3} these elements (taken over all $w\in \widehat{\Sigma}$) give a basis for $\Hom_{\catya}(\bi, \bj)$ as a left (or right) module over $H_{\bT\times\bF\times\CC^\times}^\ast(pt)$.  On the other hand, the path $\pi$ can be ``rolled'' into a diagram $D$, and Lemma \ref{lemma: where to put the Weyl group} shows that $ \morphism_D = y_w^{-1} \mathbbm{r}_\pi$. These same diagrams $D$ (over all $w\in\widehat{\Sigma}$) form a basis for $\idem(\bj) \Ya \idem(\bi)$, as shown during  the proof of \cite[Theorem 4.17]{KTWWY2}.  By the definition of our map, this basis for $\idem(\bj) \Ya \idem(\bi)$ maps bijectively onto the above basis for $\Hom_{\catya}(\bi,\bj)$, and therefore the map is an isomorphism.
\end{proof}

In \cite[Theorem 4.19]{KTWWY2}, it was proven that there is a map $Y_\mu^\lambda \rightarrow \idemCB \Ya_{\hbar = 1} \idemCB$, and claimed that this map is an isomorphism. We can complete the proof of this result:
\begin{Corollary}
The map $Y_\mu^\lambda \stackrel{\sim}{\rightarrow } \idemCB \Ya_{\hbar = 1} \idemCB$ is an isomorphism.  In finite ADE type, this upgrades to an isomorphism $\mathbf{Y}_\mu^\lambda \stackrel{\sim}{\rightarrow} \idem \Ya \idem$.
\end{Corollary}
\begin{proof}
We claim that there is a commutative diagram
\begin{equation}
\label{eq: final CD}
\begin{tikzcd}
Y_\mu^\lambda \ar[d,"\sim" labl] \ar[r] & \idemCB \Ya_{\hbar=1} \idemCB \ar[d,"\sim" labl] \\
\CBone \ar[r,"\sim"] &  \idemCB \End_{\catya}(\bi_\bv)_{\hbar = 1} \idemCB
\end{tikzcd}
\end{equation}
The top arrow is \cite[Theorem 4.19]{KTWWY2}, the left side isomorphism is from Section \ref{section: relationship to shifted Yangians}, and the right side isomorphism is from the previous theorem.  Finally, the bottom comes from the isomorphism $\CB \cong \idemCB \ICB \idemCB$ of Theorem \ref{thm: ICB matrix algebra over CB}, plus the equality $\ICB = \End_{\catya}(\bi_\bv)$ of Lemma \ref{prop: ICB in CCB}.  

To check commutativity, we can argue directly in terms of the images of the generators of the Yangian.  The element $F_i^{(r)} \in Y_\mu^\lambda$ maps to the dressed monopole operator $M_{\varpi_{i,1}, f} \in \CBone$ where $f=z_1^{r-1}$.  By Lemma \ref{lemma: minuscule monopole reexpressed}, this monopole operator maps to the endomorphism 
$$ 
\partial_{i, \bv_i} \cdots \partial_{i, 2} z_1^{r-1} r_\laG \idemCB \ \in \ \idemCB \ICB_{\hbar=1} \idemCB = \idemCB \End_{\catya}(\bi_\bv)_{\hbar=1} \idemCB
$$
Indeed, the stabilizer $W_\laG = W_{\varpi_{i,1}}$ of $(1,0,\ldots,0)$ is the symmetric group on $\{2,\ldots,\bv_i\}$, and it is easy to see that $w_0 w_\laG = s_{i,\bv_i}\cdots s_{i,1}$.  As argued in the proof of \cite[Theorem 4.19]{KTWWY2}, this endomorphism corresponds to a cylindrical KLR diagram encoding the image of $F_i^{(r)}$ in $\Ya$, which is more or less (ignoring dots, the idempotent $\idemCB$, and a sign)
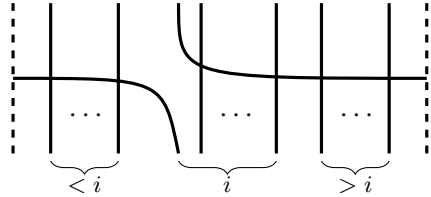
\begin{equation*}
\begin{tikzpicture}

\node at (-3,-0.5)  {$\cdots $};
\node at (0.6,-0.5) {$\cdots $};
\node at (-1,-0.5) {$\cdots$};
\draw[very thick, dashed] (-4,0) +(0,-1) -- +(0,1);
\draw[very thick, dashed] (-4,0) +(5.5,-1) -- + (5.5,1);

\draw[very thick] (-4,0) +(0.5,-1) -- +(0.5,1);
\draw[very thick] (-4,0) +(1.4,-1) -- +(1.4,1);

\draw[very thick] (-4,0) +(2.5,-1) -- +(2.5,1);
\draw[very thick] (-4,0) + (3.5,-1) -- +(3.5,1);
\draw[very thick] (-4,0) +(4.1,-1) -- +(4.1,1);
\draw[very thick] (-4,0) +(5,-1) -- +(5,1);

\draw[very thick] (-4,0) +(2.2,-1) .. controls +(-0.2,1) .. + (0,0);

\draw[very thick] (-4,0) +(5.5,0) .. controls +(-3.3,0) .. + (2.2,1);

\draw [decorate,decoration={brace,amplitude=5pt}] (-4,0) +(1.4, -1.1) -- +(0.5,-1.1)  node [black,midway,below,yshift=-2pt] {\footnotesize $<i$};
\draw [decorate,decoration={brace,amplitude=5pt}] (-4,0) +(3.5, -1.1) -- +(2.2,-1.1)  node [black,midway,below,yshift=-2pt] {\footnotesize $i$};
\draw [decorate,decoration={brace,amplitude=5pt}] (-4,0) +(5, -1.1) -- +(4.1,-1.1)  node [black,midway,below,yshift=-2pt] {\footnotesize $>i$};
\end{tikzpicture}
\end{equation*}
At the bottom, the order on the curve labels $i$ is determined by $\bi_\bv$, defined as in (\ref{eq: order and idempotent}).  Thus the image of $F_i^{(r)}$ is the same under either path in (\ref{eq: final CD}).  Similar arguments hold for the other generators of $Y_\mu^\lambda$, so the diagram is commutative.
\end{proof}

\bibliographystyle{amsalpha}
\bibliography{symmtype}

\end{document}